\documentclass[a4paper]{article}
\usepackage[utf8]{inputenc}
\usepackage[english]{babel}
\usepackage{amsthm}
\usepackage{tikz}
\usepackage{subcaption}
\usepackage{amsmath,amssymb,xcolor,latexsym,stmaryrd,url}
\usepackage{float}
\usepackage{tikz}
\usetikzlibrary{patterns}
\usepackage{pgfplots}
\pgfplotsset{compat=newest}
\newcommand{\logLogSlopeTriangle}[5]
{
    \pgfplotsextra
    {
        \pgfkeysgetvalue{/pgfplots/xmin}{\xmin}
        \pgfkeysgetvalue{/pgfplots/xmax}{\xmax}
        \pgfkeysgetvalue{/pgfplots/ymin}{\ymin}
        \pgfkeysgetvalue{/pgfplots/ymax}{\ymax}

        \pgfmathsetmacro{\xArel}{#1}
        \pgfmathsetmacro{\yArel}{#3}
        \pgfmathsetmacro{\xBrel}{#1-#2}
        \pgfmathsetmacro{\yBrel}{\yArel}
        \pgfmathsetmacro{\xCrel}{\xArel}

        \pgfmathsetmacro{\lnxB}{\xmin*(1-(#1-#2))+\xmax*(#1-#2)} % in [xmin,xmax].
        \pgfmathsetmacro{\lnxA}{\xmin*(1-#1)+\xmax*#1} % in [xmin,xmax].
        \pgfmathsetmacro{\lnyA}{\ymin*(1-#3)+\ymax*#3} % in [ymin,ymax].
        \pgfmathsetmacro{\lnyC}{\lnyA+#4*(\lnxA-\lnxB)}
        \pgfmathsetmacro{\yCrel}{\lnyC-\ymin)/(\ymax-\ymin)} % THE IMPROVED EXPRESSION WITHOUT 'DIMENSION TOO LARGE' ERROR.
        
        \coordinate (A) at (rel axis cs:\xArel,\yArel);
        \coordinate (B) at (rel axis cs:\xBrel,\yBrel);
        \coordinate (C) at (rel axis cs:\xCrel,\yCrel);

        \draw[#5]   (A)-- node[pos=0.5,anchor=north] {1}
                    (B)-- 
                    (C)-- node[pos=0.5,anchor=west] {#4}
                    cycle;
    }
}
\newcommand{\logLogSlopeTriangleinv}[5]
{
    \pgfplotsextra
    {
        \pgfkeysgetvalue{/pgfplots/xmin}{\xmin}
        \pgfkeysgetvalue{/pgfplots/xmax}{\xmax}
        \pgfkeysgetvalue{/pgfplots/ymin}{\ymin}
        \pgfkeysgetvalue{/pgfplots/ymax}{\ymax}

        \pgfmathsetmacro{\xArel}{#1}
        \pgfmathsetmacro{\yArel}{#3}
        \pgfmathsetmacro{\xBrel}{#1-#2}
        \pgfmathsetmacro{\yBrel}{\yArel}
        \pgfmathsetmacro{\xCrel}{\xBrel}

        \pgfmathsetmacro{\lnxB}{\xmin*(1-(#1-#2))+\xmax*(#1-#2)} % in [xmin,xmax].
        \pgfmathsetmacro{\lnxA}{\xmin*(1-#1)+\xmax*#1} % in [xmin,xmax].
        \pgfmathsetmacro{\lnyA}{\ymin*(1-#3)+\ymax*#3} % in [ymin,ymax].
        \pgfmathsetmacro{\lnyC}{\lnyA+#4*(\lnxA-\lnxB)}
        \pgfmathsetmacro{\yCrel}{(\lnyC-\ymin)/(\ymax-\ymin)} % THE IMPROVED EXPRESSION WITHOUT 'DIMENSION TOO LARGE' ERROR.
        
        \coordinate (A) at (rel axis cs:\xArel,\yArel);
        \coordinate (B) at (rel axis cs:\xBrel,\yBrel);
        \coordinate (C) at (rel axis cs:\xCrel,\yCrel);

        \draw[#5]   (A)-- node[pos=0.5,anchor=north] {1}
                    (B)-- node[pos=0.5,anchor=east] {#4}
                    (C)-- 
                    cycle;
    }
}

\newtheorem{assumption}{Assumption}
\newcommand{\triple}[1]{{\left\vert\kern-0.12ex\left\vert\kern-0.12ex\left\vert #1
    \right\vert\kern-0.12ex\right\vert\kern-0.12ex\right\vert}}
\DeclareMathOperator{\Div}{div}
\newcommand{\Th}{\mathcal{T}_{h}}
\newcommand{\ddn}[1]{\frac{\partial #1}{\partial n}}
\newcommand{\R}{\mathbb{R}}
\newcommand{\assign}{:=}

\usepackage{placeins}

\newtheorem{theorem}{Theorem}[section]

\newtheorem{lemma}[theorem]{Lemma}
\newtheorem{proposition}{Proposition}
\newtheorem{remark}{Remark}

\title%[$\phi$-FEM]
{A new $\phi$-FEM approach for problems with natural boundary conditions}

\author%[M. Duprez]
{Michel Duprez
\footnote{CEREMADE,  Universit\'e Paris-Dauphine \& CNRS UMR 7534, Universit\'e PSL, 75016 Paris, France.
\texttt{mduprez@math.cnrs.fr}}
and Vanessa Lleras
\footnote{IMAG, Univ Montpellier, CNRS, Montpellier, France.
\texttt{vanessa.lleras@umontpellier.fr}}
and Alexei Lozinski
\footnote{Laboratoire de Math\'ematiques de Besan\c{c}on, UMR CNRS 6623,
Universit\'e Bourgogne Franche-Comt\'e,
16, route de Gray, 25030 Besan\c{c}on Cedex,
France.
\texttt{alexei.lozinski@univ-fcomte.fr}}}

\usepackage{fancyhdr}
\begin{document}

\maketitle
% Abstract.
\begin{abstract}
We present a new finite element method, called $\phi$-FEM, to solve numerically elliptic partial differential equations with natural (Neumann or Robin) boundary conditions using simple computational grids, not fitted to the boundary of the physical domain. The boundary data are taken into account using a level-set function, which is a popular tool to deal with complicated or evolving domains. Our approach belongs to the family of fictitious domain methods (or immersed boundary methods) and is close to recent methods of cutFEM/XFEM type. Contrary to the latter, $\phi$-FEM does not need any non-standard numerical integration on cut mesh elements or on the actual boundary, while assuring the optimal convergence orders with finite elements of any degree and providing reasonably well conditioned discrete problems. In the first version of $\phi$-FEM, only essential (Dirichlet) boundary conditions was considered. Here, to deal with natural boundary conditions, we introduce the gradient of the primary solution as an auxiliary variable. This is done only on the mesh cells cut by the boundary, so that the size of the numerical system is only slightly increased . We prove theoretically the optimal convergence of our scheme and a bound on the discrete problem conditioning, independent of the mesh cuts. The numerical experiments confirm these results.
\end{abstract}

% Use the \maketitle command after the abstract

	\section{Introduction}
	
	We consider a second order elliptic partial differential equation with Neumann boundary conditions
	\begin{equation}\label{Neu:P}
	- \Delta u + u = f \text{ in } \Omega,\quad
	\frac{\partial u}{\partial n} =  0 \text{ on } \Gamma
	\end{equation}
	in a bounded domain $\Omega\subset\mathbb{R}^d$ ($d=2,3$) with smooth boundary $\Gamma$ assuming that $\Omega$ and $\Gamma$ are given by a level-set function $\phi$:
	\begin{equation}\label{eq:level-set}
	\Omega:=\{\phi<0\} \text{ and } \Gamma:=\{\phi=0\}.
	\end{equation}
	Such a representation is a popular and useful tool to deal with problems with evolving surfaces or interfaces \cite{osher}.
	In the present article, the level-set function is supposed known on $\R^d$, smooth, and to behave near $\Gamma$ similar to the signed distance to $\Gamma$.
	
	Our goal is to develop a finite element method for (\ref{Neu:P}) using a mesh which is not fitted to $\Gamma$, i.e. we allow the boundary $\Gamma$ to cut the mesh cells in an arbitrary manner. The existing finite elements methods on non-matching meshes, such as the fictitious domain/penalty method \cite{glowinski92}, XFEM \cite{moes99,moes06,sukumar2001,HaslingerRenard}, CutFEM \cite{burman2,burman15} (see also \cite{mittal2005immersed} for a review on immersed boundary methods) contain the integrals over the physical domain $\Omega$ and thus necessitate non-standard numerical integration on the parts of mesh cells cut by $\Gamma$. In this article, we propose a finite element method, based on an alternative variational formulation on an extended domain matching the computational mesh, thus avoiding any non-standard quadrature while maintaining the optimal accuracy and controlling the conditioning uniformly with respect to the position of $\Omega$ over the mesh.
	
	In the recent article \cite{phiFEM}, we have proposed such a method for the Poisson problem with homogeneous Dirichlet boundary conditions $u=0$ on $\Gamma$. The idea behind this method, baptised $\phi$-FEM, is to put $u=\phi w$ so that $u=0$ on $\Gamma $ for whatever $w$ since $\phi=0$ there. We then replace $\phi$ and $w$ by the finite element approximations $\phi_h$ and $w_h$, substitute $u\approx \phi_h w_h$ into an appropriate variational formulation and get an easily implementable discretization  in terms of the new unknown $w_h$. Such a simple idea cannot be used directly to discretize the Neumann boundary conditions in (\ref{Neu:P}). Indeed, multiplication by $\phi$ works well to strongly impose the $\textbf{essential}$ Dirichlet boundary conditions whereas Neumann conditions are $\textbf{natural}$, i.e. they come out  of the usual variational formulation without imposing them into the functional spaces.  We want thus to reformulate Problem (\ref{Neu:P}) so that Neumann conditions become essential. The way to go is the dualization of this problem, in the terminology of \cite{boffi13}, consisting in introducing an auxiliary (vector-valued) variable for the gradient $\nabla u$. In the present article, we want to use the usual conforming scalar finite elements as much as possible. Accordingly, we do not pursue the classical route of mixed methods, as in Chapter 7 of \cite{boffi13}. We shall rather introduce the additional unknowns only where they are needed, i.e. in the vicinity of boundary $\Gamma$.
	
	More specifically, let us assume that $\Omega$ lies inside a simply shaped domain $\mathcal{O}$ (typically a box in $\R^d$) and introduce a quasi-uniform simplicial mesh $\Th^{\mathcal{O}}$ on $\mathcal{O}$ (the background mesh). Let $\Th$ be a submesh of $\Th^{\mathcal{O}}$ obtained by getting rid of mesh elements lying entirely outside $\Omega$ (the definition of $\Th$ will be slightly changed afterwords). Denote by $\Omega_h$ the domain covered by mesh $\Th$ ($\Omega_h$ only slightly larger than $\Omega$) and by $\Omega_h^\Gamma$ the domain covered by mesh elements of $\Th$ cut by $\Gamma$ (a narrow strip of width $\sim h$ around $\Gamma$). Assume that the right-hand side $f$ is actually well defined on $\Omega_h$ and imagine for the moment that the solution $u$ of eq.~(\ref{Neu:P}) can be extended to a function on $\Omega_h$, still denoted by $u$, which  solves the same equation, now on $\Omega_h$: 
	\begin{equation}\label{reform1}
	- \Delta u + u = f,\quad \text{in } \Omega_h\,.
	\end{equation}
	As announced above, we now introduce an auxiliary vector-valued unknown $y$ on $\Omega_h^\Gamma$, setting $y=-\nabla u$ there, so that $u,y$ satisfy the dual form of the original equation
	\begin{equation}\label{reform2}
	y+\nabla u=0\,,
	%\end{equation}
	%\begin{equation}\label{reform3}
	\quad
	\Div y + u = f, \quad \text{in } \Omega_h^{\Gamma}\,.
	\end{equation}
	This allows us to rewrite the natural boundary condition $\frac{\partial u}{\partial n} =  0$  on $\Gamma$ as the essential condition on $y$: $y\cdot n=0$ on $\Gamma$. The latter can now be imposed using the idea of multiplication by the level-set $\phi$. To this end, we note that the outward-looking unit normal $n$ is given on $\Gamma$ by
	$
	n=\frac{1}{|\nabla\phi|}\nabla\phi\,.
	$
	Hence, we have $y\cdot n=0$ on $\Gamma$ if we put
	\begin{equation}\label{reform4}
	y\cdot \nabla\phi + p\phi  =0, \quad  \text{in } \Omega_h^\Gamma,
	\end{equation}
	where $p$ is yet another (scalar-valued) auxiliary unknown on $\Omega_h^\Gamma$. 
	
	Our finite element method, cf. (\ref{Neu:Ph}) below, will be based on a variational formulation of system (\ref{reform1})--(\ref{reform4}) treating eqs.~(\ref{reform2})--(\ref{reform4}) in a least squares manner and adding a stabilization in the vein of the Ghost penalty \cite{burmanghost}. As in \cite{phiFEM}, we coin our method $\phi$-FEM in accordance with the tradition of denoting the level-sets by $\phi$. Contrary to \cite{phiFEM}, we need here additional finite element unknowns discretizing $y$ and $p$ on $\Omega_h^\Gamma$. Since, the latter represents only a small portion of the whole computational domain $\Omega_h$, the extra cost induced by these unknowns is negligible as $h\to 0$.
	We want to emphasize that the reformulation (\ref{reform1})--(\ref{reform4}) is very formal and will serve only as a motivation for our discrete scheme (\ref{Neu:Ph}). The system (\ref{reform1})--(\ref{reform4}) itself is clearly over-determined and may well be ill-posed (the ``boundary'' conditions hidden in (\ref{reform4}) are actually not on the boundary of domain $\Omega_h$ where the problem is now posed). We shall assume neither the existence of a continuous solution to (\ref{reform1})--(\ref{reform4}), nor any properties of such a solution in the theoretical analysis of our scheme, cf. Theorem \ref{th:error}.

	The article is organized as follows: our $\phi$-FEM method is presented in the next section. We also give there the assumptions on the level-set $\phi$ and on the mesh, and announce our main result: the \textit{a priori} error estimate for $\phi$-FEM in the Neumann case. We work with standard continuous $\mathbb{P}_k$ finite elements ($k\ge 1$) on a simplicial mesh and prove the optimal order $h^k$ for the error  in the $H^1$ norm and the (slightly) suboptimal order $h^{k+1/2}$ for the error  in the $L^2$ norm. We note in passing that employing finite elements of any order is quite straightforward in our approach contrary to more traditional schemes of CutFEM type, cf. \cite{boiveau18,lehren16} for a special treatment of the case $k>1$.  The proofs of the error estimates are the subject of Section 3. Moreover, we show in Section 4 that the associated finite element matrix has the condition number of order $1/h^2$, i.e. of the same order as that of a standard finite element method on a matching grid of comparable size. In particular, the conditioning of our method does not suffer from arbitrarily bad intersections of $\Gamma$ with the mesh.  Numerical illustrations are given in Section 5. % Finally, some conclusions and perspectives are given in Section 6.

	\section{Definitions, assumptions, description of {$\phi$}-FEM, and the main result}
	
	Assume $\Omega\subset\mathcal{O}$ and let $\Th^{\mathcal{O}}$ be a quasi-uniform simplicial mesh on $\mathcal{O}$ with $h=\max_{T\in\Th}\operatorname{diam}T$ and $\rho(T)\ge \beta h$ for all $T\in\Th^{\mathcal{O}}$ with the mesh regularity parameter $\beta>0$ fixed once for all (here $\rho(T)$ is the radius of the largest ball inscribed in $T$). Fix integers $k,l \ge 1$ and let $\phi_h$ be the
	FE interpolation of $\phi$ on $\Th^{\mathcal{O}}$  by the usual continuous finite elements of degree $l$.\footnote{The integer $k$ is the degree of finite elements which will be used to approximate the principal unknown $u$ while $\phi$ is approximated by finite elements of degree $l$. We shall require $l\geq k + 1$ in our convergence Theorem \ref{th:error}. Note, that we cannot set $l=k$ unlike the Dirichlet case in \cite{phiFEM}. This is essentially due to the fact that $\phi_h$ is used here to approximate the normal on $\Gamma$ in addition to approximating $\Gamma$ itself.}
	Let $\Gamma_h:=\{\phi_h=0\}$ and introduce the computational mesh $\Th$ (approximately) covering $\Omega$ and the auxiliary mesh $\Th^\Gamma$ covering $\Gamma_h$:
	\begin{align*}
	%\label{Th}
	\mathcal{T}_h &= \{T \in \mathcal{T}_h^{\mathcal{O}} : T \cap \{\phi_h<0\} \neq
	\varnothing\}  &\text{ and }& \Omega_h = (\cup_{T \in
		\mathcal{T}_h} T)^{\circ} ,
	\\
	%\label{ThGamma}
	\mathcal{T}_h^{\Gamma} &= \{T \in \mathcal{T}_h : T \cap \Gamma_h \neq
	\varnothing%\textcolor{red}{\mbox{ or }T \cap \Gamma \neq \varnothing}
	\}  &\text{ and }& \Omega_h^{\Gamma} = (\cup_{T \in
		\mathcal{T}_h^{\Gamma}} T)^{\circ} .
	\end{align*}
	%\textcolor{red}{on a rajoute la condition en rouge pour l'estimation \eqref{eq:estim omega_h omega}}
	%Thus, $\mathcal{T}_h^{\Gamma}$ represents the cells of the mesh cut by the approximate boundary $\Gamma_h$. 
	%Note that $\Omega_h$ is typically slightly larger than $\Omega$ by a strip of the width of order $h$. In general, we do not assume however $\Omega\subset\Omega_h$ since, in some rare occasions, the boundary $\partial\Omega_h$ can be locally inside $\Omega$ due to the difference between $\phi$ and $\phi_h$.
	%\textcolor{red}{Cette precaution garantie pour moi que $B_h\subset\Omega_h^{\Gamma}$ mais peut etre il y a d'autres soucis ? Il faut commenter plus ?}
	We shall also denote by $\Omega_h^i=\Omega_h\setminus\Omega_h^\Gamma$ the domain of mesh elements completely inside $\Omega$ and set $\Gamma_h^i=\partial\Omega_h^i$. 
	%the	internal boundary of $\Omega_h^{\Gamma}$, \textit{i.e.} the ensemble of the facetsseparating $\Omega_h^{\Gamma}$ from the mesh elements inside $\Omega$, so that	$\partial \Omega_h^{\Gamma} = \partial\Omega_h \cup \Gamma_h^i$.
	
	We now introduce the finite element spaces%\footnote{We can also choose $Q_h^{(k)}$ as the continuous finite elements of degree $k-1$ on $ (\Omega^{\Gamma}_h)$ if $k \geq 2$.}
	\begin{align*}
	V_h^{(k)} &= \{v_h \in H^1 (\Omega_h) : v_h |_T \in \mathbb{P}_k (T) \quad \forall
	T \in \mathcal{T}_h \},\\
	Z_h^{(k)} &= \{z_h \in H^1 (\Omega^{\Gamma}_h)^d : z_h |_T \in \mathbb{P}_k (T)^d
	\quad \forall T \in \mathcal{T}_h^{\Gamma} \},\\
	Q_h^{(k)} &= \{q_h \in L^2 (\Omega^{\Gamma}_h) : q_h |_T \in
	\mathbb{P}_{k - 1} (T) \quad \forall T \in \mathcal{T}_h^{\Gamma} \},\\
	W_h^{(k)} &= V_h^{(k)} \times Z_h^{(k)} \times Q_h^{(k)}
	\end{align*}
	and the finite element problem: Find $(u_h, y_h,p_h) \in
	W_h^{(k)}$ such that
	\begin{align}\label{Neu:Ph}
	a_h (u_h, y_h, p_h ; v_h, z_h, q_h) = &\int_{\Omega_h} fv_h + \gamma_{div}
	\int_{\Omega_h^{\Gamma}} f( \Div z_h+v_h)  ,
	\end{align}
	for all $(v_h, z_h, q_h) \in W_h^{(k)}$, where
	\begin{multline*}
	%\label{Rob:ah}
	a_h (u, y, p ; v, z, q) = \int_{\Omega_h} \nabla u \cdot
	\nabla v + \int_{\Omega_h} uv + \int_{\partial\Omega_h} y \cdot nv\\
	+ \gamma_{div}  \int_{\Omega_h^{\Gamma}} (\Div y+u)( \Div z+v) 
	+ \gamma_u  \int_{\Omega_h^{\Gamma}} (y + \nabla u) \cdot (z + \nabla v) \\
	+{\frac{\gamma_p}{h^2}}  \int_{\Omega_h^{\Gamma}} (y \cdot  \nabla \phi_h 
	+ \frac{1}{h}p \phi_h)  (z \cdot \nabla \phi_h + \frac{1}{h}q \phi_h)
	+ \sigma h \int_{\Gamma_h^i} \left[ \ddn{u} \right] \left[ \ddn{v} \right] 
	\end{multline*}
	with some positive numbers $\gamma_{div}$, $\gamma_u$, $\gamma_p$, and $\sigma$
	properly chosen in a manner independent of $h$. We have assumed here that $f$ is well defined on $\Omega_h$, rather than on $\Omega$ only. 
	%Both $f$ and $\tilde{g}$ are supposed sufficiently smooth as detailed in the forthcoming statement of our main theorem.
	
	The finite element problem (\ref{Neu:Ph}) is inspired by (\ref{reform1})--(\ref{reform4}). The first line in the definition of $a_h$ comes from multiplying (\ref{reform1}) by a test function $v$, integrating by parts
	\[ \int_{\Omega_h} \nabla u \cdot \nabla v + \int_{\Omega_h} uv -
	\int_{\partial \Omega_h} \nabla u \cdot nv = \int_{\Omega_h} fv \]
	and noting that $- \nabla u \cdot n = y \cdot n$ on $\partial \Omega_h$ by (\ref{reform2}). Equations (\ref{reform2})--(\ref{reform4}) are than added in least squares manner, introducing the test functions $z$ and $q$ corresponding to $y$ and $p$ respectively. Note that we replace $p$ by $\frac{1}{h}p$ in the term stemming from (\ref{reform4}). This rescaling does not affect the discretization of $u$ (which is the only quantity that interests us) and will be crucial to control the conditioning of the method. Finally, the terms multiplied by $\sigma h$ is the Ghost penalty from \cite{burmanghost} (we need to penalize the jumps only on $\Gamma_h^i$ because some continuity of $\nabla u_h$ on the facets inside $\Omega_h^\Gamma$ is already enforced by assimilating $\nabla u_h$ to $y_h$ which is continuous).

	We now recall some technical assumptions on the domain and the mesh, the same as in \cite{nocut, phiFEM}.  These assumptions hold true for smooth domains and sufficiently refined meshes.

	\begin{assumption}\label{asm0}
		%\al{
		There exists a neighborhood of $\Gamma$, a domain $\Omega^\Gamma$, which can be covered by open sets $\mathcal{O}_i$, $i=1,\ldots,I$ and one can introduce on every $\mathcal{O}_i$ local coordinates $\xi_1,\ldots,\xi_d$ with $\xi_d=\phi$ such that all the partial derivatives $\partial^\alpha\xi/\partial x^\alpha$ and $\partial^\alpha x/\partial \xi^\alpha$ up to order $k+1$ are bounded by some $C_0>0$. Thus, $\phi$ is of class $C^{k+2}$ on $\Omega^\Gamma$. Moreover, $|\nabla\phi|\ge m$ on $\Omega^\Gamma$ with some $m>0$.
	\end{assumption}
	
	\begin{assumption}\label{asm1}% \al{
		$\Omega_h^\Gamma\subset\Omega^\Gamma$ and $|\nabla\phi_h|\ge \frac m 2$ on all the mesh elements of $\Omega_h^\Gamma$.% }
	\end{assumption}
	%This assumption is clearly valid for $h$ small enough. In particular, it guarantees that $\Gamma_h$ is indeed a curve in 2D (a surface in 3D). Note that it could be in principle possible that $\phi$ vanished on all the interpolation points used for $\phi_h$ on an element $T$ so that $\phi_h$ would vanish on the hole element $T$ which would be thus included into $\Gamma_h$. Assuming that $|\nabla\phi_h|>0$ excludes these situations (which are highly unlikely any way).

	\begin{assumption}\label{asm2}
		The approximate boundary $\Gamma_h$ can be covered by element patch\-es $\{\Pi_k \}_{k = 1,
			\ldots, N_{\Pi}}$ having the following properties:
		\begin{itemize}
			\item Each $\Pi_k$ is composed of a mesh element $T_k$ lying inside $\Omega$ and some elements cut by $\Gamma$, more precisely $\Pi_k = T_k \cup \Pi_k^{\Gamma}$ where $T_k\in\mathcal{T}_h$, $T_k\subset\bar\Omega$, $\Pi_k^{\Gamma}\subset\mathcal{T}_h^{\Gamma}$, and $\Pi_k^{\Gamma}$ contains at most $M$ mesh elements;
			\item Each mesh element in a patch $\Pi_k$ shares at least a facet with another mesh element in the same patch. In particular, $T_k$ shares a facet $F_k$ with an element in $\Pi_k^\Gamma$;
			\item $\mathcal{T}_h^{\Gamma} = \cup_{k = 1}^{N_{\Pi}} \Pi_k^{\Gamma}$
			and $\Gamma_h^{i} = \cup_{k = 1}^{N_{\Pi}} F_k$;
			\item $\Pi_k$ and $\Pi_l$ are disjoint if $k \neq l$.
		\end{itemize}
	\end{assumption}
	Assumption \ref{asm2} prevents strong oscillations of $\Gamma$ on the length scale $h$. It can be reformulated by saying that each cut element $T\in\Th^{\Gamma}$ can be connected to an uncut element $T'\subset\Omega_h^i$ by a path consisting of a small number of mesh elements adjacent to one another;  see \cite{nocut} for a more detailed discussion and an illustration (Fig.~2).

	\begin{theorem}\label{th:error}
		Suppose that Assumptions \ref{asm0}--\ref{asm2} hold true, $l\ge k+1$, $\Omega\subset\Omega_h$ and $f \in H^k (\Omega_h)$. Let $u \in H^{k + 2} (\Omega)$  be the solution to \eqref{Neu:P} and $(u_h,y_h,p_h)\in W_h^{(k)}$ be the solution to \eqref{Neu:Ph}. Provided $\gamma_{div}$, $\gamma_u$, $\gamma_p$, $\sigma$ are sufficiently big, it holds
		\begin{align} \label{Neu:aprio}
		|u - u_h |_{1, \Omega} \le Ch^k  \|f\|_{k,\Omega_h} 
		%\\
		\mbox{ ~~and~~ }
		\|u - u_h \|_{0, \Omega} \le Ch^{k+1/ 2} \|f\|_{k, \Omega_h}  %\label{Neu:aprioL2}   
		\end{align}
		with  $C > 0$ depending on the constants in Assumptions \ref{asm0}, \ref{asm2} (and thus on the norm of $\phi$ in $C^{k+2}$), on the mesh regularity, on the polynomial degrees $k$ and $l$, and on $\Omega$, but independent of $h$, $f$, and $u$.
	\end{theorem}

	\begin{remark}[(Condition $\Omega\subset\Omega_h$)]
		The assumptions of Theorem \ref{th:error} include $\Omega\subset\Omega_h$. Note that one would automatically have $\Omega\subset\Omega_h$, were $\Omega_h$ defined as the set of mesh cells having a non empty intersection with $\Omega=\{\phi<0\}$. However, $\Omega_h$ is based on the intersections with $\{\phi_h<0\}$ which can result in some rare situation where tiny portions of $\Omega$ lie outside $\Omega_h$. In such a case, the \textit{a priori} estimates (\ref{Neu:aprio})%--(\ref{Neu:aprioL2})
		will control the error only on $\Omega\cap\Omega_h$.
	\end{remark}

%	\begin{remark}\label{nonhom}
%			We can also treat the case of non-homogeneous boundary conditions $\frac{\partial u}{\partial n}=g$ on $\Gamma$ by adding the term 
%			$$ - {\frac{\gamma_p}{h^2}}
%			\int_{\Omega_h^{\Gamma}} \tilde{g} | \nabla \phi_h |  (z_h \cdot \nabla
%			\phi_h + \frac{1}{h}q_h \phi_h) $$
%			in the right-hand side  of \eqref{Neu:Ph} where $\tilde{g}\in H^{k+1}(\Omega_h^{\Gamma})$ is lifting of $g$ from $\Gamma$ to a vicinity of $\Gamma$. Theorem \ref{th:error} remains valid, adding  $\| \tilde{g} \|_{k
%				+ 1, \Omega_h^{\Gamma}}$ to $\|f\|_{k,\Omega_h}$ in \eqref{Neu:aprio}.
%	\end{remark}	

		\begin{remark}[(non-homogeneous Neumann and Robin conditions)]\label{nonhom}
			We can also treat the case of more general boundary conditions:
			\begin{enumerate}
			\item[(i)] non-homogeneous Neumann boundary conditions $\frac{\partial u}{\partial n}=g$ on $\Gamma$ by adding the term 
			$$ - {\frac{\gamma_p}{h^2}}
			\int_{\Omega_h^{\Gamma}} \tilde{g} | \nabla \phi_h |  (z_h \cdot \nabla
			\phi_h + \frac{1}{h}q_h \phi_h) $$
			in the right-hand side  of \eqref{Neu:Ph} where $\tilde{g}\in H^{k+1}(\Omega_h^{\Gamma})$ is lifting of $g$ from $\Gamma$ to a vicinity of $\Gamma$.
			\item[(ii)] Robin boundary condition $\frac{\partial u}{\partial n}+\alpha u=g$ on $\Gamma$ ($\alpha\in\mathbb{R}$) by replacing the penultimate term in $a_h$ by
			\begin{equation*}
		{\frac{\gamma_p}{h^2}}  \int_{\Omega_h^{\Gamma}} (y \cdot  \nabla \phi_h -| \nabla \phi_h |\alpha u + \frac{1}{h}p \phi_h)  (z \cdot \nabla \phi_h -| \nabla \phi_h |\alpha v + \frac{1}{h}q \phi_h)
	\end{equation*}
		and by adding the term
		$$ - {\frac{\gamma_p}{h^2}}
			\int_{\Omega_h^{\Gamma}} \tilde{g} | \nabla \phi_h |  (z_h \cdot \nabla
			\phi_h -| \nabla \phi_h | \alpha v+\frac{1}{h}q_h \phi_h) $$		
	in the right-hand side  of \eqref{Neu:Ph} where $\tilde{g}\in H^{k+1}(\Omega_h^{\Gamma})$ is defined as before.			
			\end{enumerate}			
			 Theorem \ref{th:error} remains valid, adding  $\| \tilde{g} \|_{k
				+ 1, \Omega_h^{\Gamma}}$ to $\|f\|_{k,\Omega_h}$ in \eqref{Neu:aprio}.
				This framework will be used in first test case of the numerical simulations performed in Section \ref{sec:num}: Fig. 2-8 for (i) and Fig. 9 for (ii).
	\end{remark}

\section{Proof of the \textit{a priori} error estimates}

From now on, we shall use the letter $C$ for positive constants (which can vary from one line to another) that depend only on the regularity of the mesh and on  the constants in Assumptions \ref{asm0}--\ref{asm2}.

%\al{
We shall begin with some technical results, mostly adapted from \cite{nocut} and \cite{phiFEM} to be used later in the proofs of the coercivity of $a_h$ (Section 3.2) and the \textit{a priori} error estimates (Sections 3.3 and 3.4).  
%The most important contribution here is Lemma \ref{LemNeu:prop1} which extends to finite elements of any degree a result from \cite{nocut}. This lemma will be the keystone of the proof of coercivity by allowing us to handle the non necessarily positive terms on the cut elements. 
%The following two Lemmas 3.4 and 3.5 present appropriately scaled versions of trace and Poincar\'e inequalities on the strip $\Omega_h^\Gamma$ and will prove themselves useful on several occasions. The last  Lemma 3.6 will help us to handle the extension of the exact solution $u$ to the ``fictitious'' domain $\Omega_h$, cf (\ref{ftildef}). 
%}

\subsection{Technical lemmas}

%We recall first a lemma from \cite{nocut}
%
%\begin{lemma}  \label{LemNeu:prop1}
%Assume that Assumption \ref{asm2} holds.
%  For any $\beta > 0$, there exist $0 < \alpha < 1$ and
%  $\delta > 0$ depending only on the mesh regularity and geometrical
%  assumptions such that
%  \begin{equation}
%    \left| \int_{\Omega^{\Gamma}_h} z_h \cdot \nabla v_h \right| \leq
%    \alpha |v|_{h 1, \Omega_h}^2 + \delta \|z_h + \nabla v_h \|^2_{0,
%    \Omega_h^{\Gamma}} + \beta h \left\| \left[ \ddn{v_h} \right] \right\|_{0,
%    \Gamma_h^i}^2, \label{Neu:prop1}
%  \end{equation}
%  for all $v_h \in V_h^{(k)}, z_h \in Z_h^{(k)}$.
%\end{lemma}

We recall first a lemma from \cite{phiFEM}:

\begin{lemma}
  \label{lemma:poly}Let $T$ be a triangle/tetrahedron, $E$ one of its sides
  and $p$ a polynomial on $T$ such that $p = a$ on $E$ for some $a \in \mathbb{R}$,
  $\frac{\partial p}{\partial n} = 0$ on $E$, and $\Delta p  = 0$
  on $T$. Then $p = a$ on  $T$.
\end{lemma}
%\textcolor{red}{je suis revenu au lemme de [phiFEM] exactement, parce que ma modification da la version precedente etait fausse}

We now adapt a lemma from \cite{nocut}:

\begin{lemma}
%  \label{LemNeu:prop1}
Let $B_h$ be the strip  between $\partial\Omega_h $ and $\Gamma_h$.   For any $\beta > 0$, there exist $0 < \alpha < 1$ and $\delta > 0$ depending only on the mesh regularity and geometrical  assumptions such that, for all $v_h \in V_h^{(k)}, z_h \in Z_h^{(k)}$
  \begin{equation}
    \left| \int_{B_h} z_h \cdot \nabla v_h \right| \leq
    \alpha |v_h |_{1, \Omega_h}^2 + \delta \|z_h + \nabla v_h \|^2_{0,
    \Omega_h^{\Gamma}} + \beta h \left\| \left[ \ddn{v_h} \right] \right\|_{0,
    \Gamma_h^i}^2 
    + \beta h^2 \| \Div z_h +v_h \|_{0,  \Omega_h^{\Gamma}}^2
    + \beta h^2 \| v_h  \|_{0,  \Omega_h^{\Gamma}}^2 .
    \label{Neu:prop1}
  \end{equation}
\end{lemma}
%\textcolor{red}{Notez le changement dans le terme avec $\Div z_h$ qui remplace $\Div z_h+v_h$. Cela entrainera des modifications dans la preuve du lemme de coercivite, mais c'est facilement gerable. }

\begin{proof}
  The boundary $\Gamma$ can be covered by element patches $\{\Pi_k \}_{k = 1,
  \ldots, N_{\Pi}}$ as in Assumption \ref{asm2}. Choose any $\beta > 0$ and
  consider
\begin{equation}
  \label{Neu:alph} \alpha \assign \max_{\Pi_k, (z_h, v_h) \neq (0, 0)} F
  (\Pi_k, z_h, v_h)
\end{equation}
with
\begin{equation*}
 F (\Pi_k, z_h, v_h) = \frac{\|z_h \|_{0, \Pi_k^{\Gamma}} |v_h |_{1,
   \Pi_k^{\Gamma}} - \beta \|z_h + \nabla v_h \|^2_{0, \Pi_k^{\Gamma}} - \beta
   h \left\| \left[ \ddn{v_h} \right] \right\|^2_{0, F_k} - \frac \beta 2 h^2  \| \Div z_h  \|^2_{0,
   \Pi_k^{\Gamma}}}{\frac{1}{2} \|z_h \|_{0, \Pi_k^{\Gamma}}^2 + \frac{1}{2}
   |v_h |_{1, \Pi_k}^2}, 
   \end{equation*}
where the maximum is taken over all the possible configurations of a patch
$\Pi_k$ allowed by the mesh regularity and over all $v_h \in V_h^{(k)}$ and
$z_h \in Z_h^{(k)}$ restricted to $\Pi_k$.
%, i.e. $v_h \in H^1 (\Pi_k)$, $v_h \in \mathbb{P}_k (T)$ (resp. $z_h \in H^1 (\Pi^{\Gamma}_k)^d$, $z_h \in \mathbb{P}_k (T)^d$) for all mesh elements $T \in \Pi_k$ (resp. $T \in \Pi_k^{\Gamma}$). 
Note that $F (\Pi_k, z_h, v_h)$ is invariant under the
scaling transformation $x \mapsto \frac{1}{h} x$, $v_h \mapsto \frac{1}{h}v_h$, $z_h \mapsto z_h$. 
%This means that if we introduce $\mathcal{S} (x) = \frac{1}{h}x$ and set $\tilde{\Pi} =\mathcal{S} (\Pi_k)$, $\tilde{\Pi}^{\Gamma}=\mathcal{S} (\Pi^{\Gamma}_k)$, $\tilde{F} =\mathcal{S} (F_k)$, $\tilde{T} =\mathcal{S} (T)$, \ $\tilde{z} = z_h \circ \mathcal{S}^{- 1}$, $\tilde{v} = \frac{1}{h} v_h \circ\mathcal{S}^{- 1}$ then
%\[ F (\Pi_k, z_h, v_h) = \frac{\| \tilde{z} \|_{0, \tilde{\Pi}^{\Gamma}} |
%   \tilde{v} |_{1, \tilde{\Pi}^{\Gamma}} - \beta \| \tilde{z} + \nabla
%   \tilde{v} \|^2_{0, \tilde{\Pi}^{\Gamma}} - \beta \left\| \left[
%   \ddn{\tilde{v}} \right] \right\|^2_{0, \tilde{F}} - \frac \beta  2\| \Div\tilde{z}
%   \|^2_{0, \tilde{\Pi}^{\Gamma}}}{\frac{1}{2} \| \tilde{z} \|_{0,
%   \tilde{\Pi}^{\Gamma}}^2 + \frac{1}{2} | \tilde{v} |_{1, \tilde{\Pi}^{}}^2}
%   . \]
We can thus assume $h = 1$ when computing the maximum in (\ref{Neu:alph}).
Moreover, $F (\Pi_k, z_h, v_h)$ is homogeneous with respect to $v_h$, $z_h$, i.e.
$F (\Pi_k, z_h, v_h) = F (\Pi_k, \mu z_h, \mu v_h) $ for any $\mu \neq 0$. Thus, the maximum in (\ref{Neu:alph}) is indeed attained since it can be taken over a closed
bounded set in a finite dimensional space (all the admissible patches on a
mesh with $h = 1$ and all $v_h, z_h$ such that $|v_h |_{1, \Pi_k}^2 + \|z_h
\|_{0, \Pi_k^{\Gamma}}^2 = 1$).

  Clearly, $\alpha \leq 1$. Supposing $\alpha = 1$ leads to a
  contradiction. Indeed, if $\alpha = 1$, we can then take $\Pi_k$, $v_h$,
  $z_h$ yielding this maximum (in particular, $|v_h |_{1, \Pi_k}^2 + \|z_h
  \|_{0, \Pi_k^{\Gamma}}^2 > 0$). We observe then
  \begin{equation*}
   \frac{1}{2} |v_h |_{1, \Pi_k}^2 - \|z_h \|_{0, \Pi_k^{\Gamma}} |v_h |_{1,
     \Pi_k^{\Gamma}} + \frac{1}{2} \|z_h \|_{0, \Pi_k^{\Gamma}}^2 + \beta
     \|z_h + \nabla v_h \|^2_{0, \Pi_k^{\Gamma}} + \beta h \left\| \left[
     \ddn{v_h} \right] \right\|_{0, F_k}^2 + \frac \beta 2 h^2 \| \Div z_h \|^2_{0,
     \Pi_k^{\Gamma}} = 0
     \end{equation*}
  and consequently (recall $|v_h |_{1, \Pi_k}^2 = |v_h |_{1, T_k}^2 + |v_h
  |_{1, \Pi_k^{\Gamma}}^2$)
\begin{equation}
  \label{Neu:alph1} \frac{1}{2} |v_h |_{1, T_k}^2  + \beta \|z_h + \nabla v_h
  \|^2_{0, \Pi_k^{\Gamma}} + \beta h \left\| \left[ \ddn{v_h} \right]
  \right\|_{0, F_k}^2 + \frac \beta 2 h^2  \|
  \Div z_h  \|^2_{0, \Pi_k^{\Gamma}} = 0.
\end{equation}
This implies
$|v_h |_{1, T_k}^{}=0$ so that $v_h =\text{const}$ on $T_k$. Moreover, $\|z_h + \nabla v_h \|_{0, \Pi_k^{\Gamma}}=0$ so that $\nabla v_h = -
  z_h$ on $\Pi_k^{\Gamma}$, hence $\nabla v_h$ is continuous on
  $\Pi_k^{\Gamma}$ and $\Delta v_h=0$ on  $\Pi_k^{\Gamma}$ since $\Div z_h  = 0$  there. The jump $\left[ \ddn{v_h} \right]$ vanishes also on the facet $F_k$ separating $T_k$ from $\Pi_k^{\Gamma}$, as implied directly by  (\ref{Neu:alph1}). Combining these observations with Lemma \ref{lemma:poly}, starting from $T_k$ and its neighbor in $\Pi^{\Gamma}_k$ and then propagating to other elements of $\Pi^{\Gamma}_k$, we see that $v_h =\text{const}$ on the whole $\Pi_k$. We have thus $\nabla v_h = 0$ on $\Pi_k$ and $z_h = 0$ on $\Pi_k^{\Gamma}$, which is in contradiction with $|v_h |_{1, \Pi_k}^2 + \|z_h \|_{0, \Pi_k^{\Gamma}}^2 > 0$.

  Thus $\alpha < 1$ and
  \begin{equation*}
    \|z_h \|_{0, \Pi_k^{\Gamma}} |v_h |_{1, \Pi_k^{\Gamma}} \leq
    \frac{\alpha}{2} \|z_h \|_{0, \Pi_k^{\Gamma}}^2 + \frac{\alpha}{2} |v_h
    |_{1, \Pi_k}^2 + \beta \|z_h + \nabla v_h \|^2_{0, \Pi_k^{\Gamma}} + \beta
    h \left\| \left[ \ddn{v_h} \right] \right\|^2_{0, \partial T_k \cap \partial
    \Pi_k^{\Gamma}} + \frac \beta 2 h^2 \| \Div z_h \|_{0,
    \Pi_k^{\Gamma}}^2
  \end{equation*}
  for all $v_h, z_h$ and all admissible patches $\Pi_k$. We now observe
  \begin{multline*}
    \left| \int_{B_h} z_h \cdot \nabla v_h \right|  \leq
    \sum_k \left| \int_{B_h \cap \Pi_k^{\Gamma}} z_h \cdot \nabla v_h \right|
    \leq \sum_k \|z_h \|_{0, \Pi_k^{\Gamma}} |v_h |_{1, \Pi_k^{\Gamma}}\\
    \leq \frac{\alpha}{2} \|z_h \|_{0, \Omega_h^{\Gamma}}^2 +
    \frac{\alpha}{2} |v_h |_{1, \Omega_h}^2 + \beta \|z_h + \nabla v_h
    \|^2_{0, \Omega_h^{\Gamma}} + \beta h \left\| \left[ \ddn{v_h} \right]
    \right\|^2_{0, \Gamma_h^i}+ \frac \beta 2 h^2 \| \Div z_h \|_{0, \Omega_h^{\Gamma}}^2.
  \end{multline*}
  We now use  the Young inequality with any $\varepsilon > 0$ to obtain
  \[ \|z_h \|_{0, \Omega_h^{\Gamma}}^2 = \|z_h + \nabla v_h \|_{0,
     \Omega_h^{\Gamma}}^2 + \| \nabla v_h \|_{0, \Omega_h^{\Gamma}}^2 - 2 (z_h
     + \nabla v_h, \nabla v_h)_{0, \Omega_h^{\Gamma}} 
   \leq \left( 1 + \frac{1}{\varepsilon} \right)  \|z_h + \nabla v_h
     \|_{0, \Omega_h^{\Gamma}}^2 + (1 + \varepsilon) |v_h |_{1, \Omega_h}^2, \]
  which leads to
  \begin{equation*}\left| \int_{B_h} z_h \cdot \nabla v_h \right|
   \leq \alpha \left( 1 + \frac{\varepsilon}{2} \right) |v_h |_{1,
    \Omega_h}^2 + \left( \beta + \frac{\alpha}{2} + \frac{\alpha}{2
    \varepsilon} \right)  \|z_h + \nabla v_h \|^2_{0, \Omega_h^{\Gamma}} +
    \beta h \left\| \left[ \ddn{v_h} \right] \right\|^2_{0, \Gamma_h^i} + \beta
    h^2 \| \Div z_h  \|_{0, \Omega_h^{\Gamma}}^2.
  \end{equation*}
  Taking $\varepsilon$ sufficiently small, redefining $\alpha$ as $\alpha
  \left( 1 + \frac{\varepsilon}{2} \right)$ and putting $\delta = \left( \beta
  + \frac{\alpha}{2} + \frac{\alpha}{2 \varepsilon} \right)$ we obtain
  $$
      \left| \int_{B_h} z_h \cdot \nabla v_h \right| \leq
    \alpha |v_h |_{1, \Omega_h}^2 + \delta \|z_h + \nabla v_h \|^2_{0,
    \Omega_h^{\Gamma}} + \beta h \left\| \left[ \ddn{v_h} \right] \right\|_{0,
    \Gamma_h^i}^2 
    +  \beta  h^2 \| \Div z_h \|_{0,  \Omega_h^{\Gamma}}^2 .
  $$
  This leads to (\ref{Neu:prop1}) by the triangle inequality
  $ \| \Div z_h \|_{0,  \Omega_h^{\Gamma}} \le \| \Div z_h +v_h \|_{0,\Omega_h^{\Gamma}} 
    + \| v_h  \|_{0,\Omega_h^{\Gamma}} $.
\end{proof}

\begin{lemma}\label{Prelim1}
  For all $v \in H^1 (\Omega_h^{\Gamma})$,
 $
    \|v\|_{0, \Omega_h^{\Gamma}} \le C \left( \sqrt{h} \|v\|_{0, \Gamma_h^i} +
    h|v|_{1, \Omega_h^{\Gamma}} \right)% \label{Dir:eq:1}
$\\
%\textcolor{red}{
and for all $v \in H^1 (\Omega_h\backslash \Omega)$,
  $
    \|v\|_{0, \Omega_h\backslash \Omega} \le C \left( \sqrt{h} \|v\|_{0, \Gamma} +
    h|v|_{1, \Omega_h\backslash \Omega} \right).% \label{Dir:eq:1}
$
  %}
\end{lemma}
We refer to \cite{nocut} for the first inequality. The second one can be treated similarly. 

The following lemma is borrowed from \cite{phiFEM}. It's a partial generalization of Lemma \ref{Prelim1} to derivatives of higher order.  
\begin{lemma}\label{Prelim15}
Under Assumption \ref{asm0}, it  holds for all $v \in H^{s} (\Omega_h)$ with integer $1\le s\le k+1$,
$v$ vanishing on  $\Omega$,
$%\label{Dir:eq:4}
\left\|v \right\|_{0,\Omega_h\setminus\Omega}
\le  Ch^s\left\|v\right\|_{s, \Omega_h\setminus\Omega}.
$ 
\end{lemma}

\begin{lemma}\label{Prelim3}
  For all piecewise polynomial (possibly discontinuous) functions $v_h$ on $\Th^\Gamma$, $\quad 
    \|v_h \|_{0, \Gamma_h} \leq \frac {C} {\sqrt{h}} \|v_h \|_{0,\Omega_h^{\Gamma}}
$
with a constant $C>0$ depending on the maximal degree of polynomials in $v_h$ and on the constants in Assumptions \ref{asm0}--\ref{asm2}.
\end{lemma}
\begin{proof} A scaling argument on all $T \in \Th^{\Gamma}$.
%Take any $T \in \Th^{\Gamma}$ and denote $E = T \cap \Gamma_h$.
%%We have by Assumption \ref{asm1}
%%\textcolor{red}{
%Since $\phi_h$ is a polynomial of degree $l$ on each cell
% $T$, a scaling argument gives $
%|\Gamma_h\cap T|\le\frac{C_\Gamma}{h}|T|
%$
%with some $C_\Gamma>0$ depending on $l$ and the mesh regularity. Hence
%%}
%$
%  %\label{eq:trace}
%  \|v_h \|^2_{0, E} \le \|v_h \|^2_{L^{\infty} (T)} | E |
%  \leq \frac{C_{\Gamma}}{h} \|v_h \|^2_{L^{\infty} (T)} | T |.
%$
%Applying the inverse inequality $\|v_h \|^2_{L^{\infty} (T)} \leq C \|v_h
%\|^2_{0, T} / | T |$ yields
%$ \|v_h \|^2_{0, E} \leq \frac{C}{h} \|v_h \|^2_{0, T} . $
%Summing over all $T \in \Th^{\Gamma}$ concludes the proof.
\end{proof}

%\begin{lemma}\label{Prelim4}\textcolor{red}{pas utilise pour le moment}
%  For all $v_h \in V_h^{(s)}$ and $z_h \in Z_h^{(s)}$, $s\geq 0$,
%  \begin{equation*}
%\| v_h \|_{0, \Omega_h^{\Gamma}}^2\leq C\left(\| v_h \|_{0, \Omega_h^i}^2 + h^4\| \Div
%    z_h\|^2_{0, \Omega_h^{\Gamma}} +  h^2\|z_h + \nabla v_h\|^2_{0,
%    \Omega_h^{\Gamma}} + h^3 \left\| \left[ \ddn{v}_h \right] \right\|^2_{0,
%    \Gamma_h^i}\right)
%  \end{equation*}
%\end{lemma}

Finally, we recall a Hardy-type lemma, cf. \cite{phiFEM}. 
%Unlike the Dirichlet case in \cite{phiFEM}, we need the level-set $\phi$ only in a neighborhood $\Omega^\Gamma$ of $\Gamma$. The lemma is adapted accordingly.
\begin{lemma}
  \label{lemma:hardy}Assume that the domain $\Omega^\Gamma$ is a neighborhood of $\Gamma$,
 given by (\ref{eq:level-set}), and satisfies Assumption \ref{asm0}. Then, for any $u \in H^{s + 1} (\Omega^\Gamma)$ vanishing on $\Gamma$ and an integer $s\in[0,k]$, it holds
$ \left\| \frac{u}{\phi} \right\|_{s, \Omega^\Gamma} \le C \| u
     \|_{s + 1, \Omega^\Gamma}$  with $C > 0$ depending only on the constants in Assumption \ref{asm0} and on $s$. 
\end{lemma}

\subsection{Coercivity of the bilinear form $a$}

It will be convenient to rewrite the bilinear form $a_h$ in a manner avoiding
the integral on $\partial \Omega_h$. To this end, we recall that $B_h$ is the
strip between $\partial \Omega_h$ and $\Gamma_h$ and observe for any $y \in
H^1 (B_h)^d$, $v \in H^1 (B_h)$, $q \in L^2 (\Gamma_h)$:
\[ \int_{\partial \Omega_h} y \cdot nv_{} = \int_{\partial \Omega_h} y \cdot
   nv - \int_{\Gamma_h} \frac{1}{| \nabla \phi_h |}  (y \cdot \nabla \phi_h)
   v_{} + \int_{\Gamma_h} \frac{1}{| \nabla \phi_h |}  (y \cdot \nabla \phi_h
   +  \frac{1}{h}q \phi_h) v \]
\[ = \int_{B_h} (v_{} \Div y + y \cdot \nabla v) + \int_{\Gamma_h} \frac{1}{|
   \nabla \phi_h |}  (y \cdot \nabla \phi_h +  \frac{1}{h}q \phi_h) v . \]
Indeed, $\phi_h = 0$ on $\Gamma_h$ and the unit normal to $\Gamma_h$, looking
outward from $B_h,$ is equal to $- \nabla \phi_h / | \nabla \phi_h |$. Thus,
\begin{multline} \label{altah} 
 a_h (u, y, p ; v, z, q) = \int_{\Omega_h} \nabla u \cdot
  \nabla v + \int_{\Omega_h} uv + \int_{B_h} (v_{} \Div y + y \cdot \nabla v)\\
  + \int_{\Gamma_h} \frac{1}{| \nabla \phi_h |}  (y \cdot \nabla \phi_h  + \frac{1}{h}q  \phi_h) v 
  + \gamma_{div}  \int_{\Omega_h^{\Gamma}} (\Div y + u)  (\Div z + v)
  + \gamma_u  \int_{\Omega_h^{\Gamma}} (y + \nabla u) \cdot (z + \nabla v) \\
  +  \sigma h \int_{\Gamma_h^i} \left[ \ddn{u} \right] \left[ \ddn{v} \right] 
  + \frac{\gamma_p}{h^2}  \int_{\Omega_h^{\Gamma}} (y \cdot \nabla \phi_h  + \frac{1}{h} p \phi_h)  (z \cdot \nabla \phi_h + \frac{1}{h}q \phi_h).
\end{multline}

\begin{proposition}
  \label{LemRob:coer}Provided $\gamma_{div}, \gamma_u, \gamma_p, \sigma$ are
  sufficiently big, there exists an $h$-independent constant $c > 0$ such that
  \begin{eqnarray*}
    a_h (v_h, z_h, q_h ; v_h, z_h, q_h) \ge c \triple{ v_h, z_h, q_h }_h^2,
    \quad \forall (v_h, z_h , q_h) \in W_h^{(k)}
  \end{eqnarray*}
  with
   \begin{equation*}
   \triple{ v_{}, z_{}, q_{} }_h^2 = \| v \|_{1, \Omega_h}^2 + \| \Div
    z+v\|^2_{0, \Omega_h^{\Gamma}} +  \|z + \nabla v\|^2_{0,
    \Omega_h^{\Gamma}}+ h \left\| \left[ \ddn{v} \right] \right\|^2_{0,
    \Gamma_h^i} 
    + \frac{1}{h^2} \left\| z \cdot \nabla \phi_h + \frac{1}{h}q \phi_h \right\|_{0, \Omega_h^{\Gamma}}^2.
   \end{equation*} 
\end{proposition}

\begin{proof}
%  Let $\Gamma_h = \{ \phi_h = 0 \}$, \textit{i.e.} the approximation of the actual
%  boundary $\Gamma$ given by the approximate level-set.
  Using the reformulation of the bilinear form $a_h$ given by (\ref{altah}), we have for all $(v_h, z_h,q_h) \in W_h^{(k)}$,
  \begin{multline*}
    a_h (v_h, z_h, q_h ; v_h, z_h, q_h) =  |v_h|_{1, \Omega_h}^2 + \| v_h \|_{0, \Omega_h}^2 +
    \int_{B_h} (v_h \Div z_h + z_h \cdot \nabla v_h) \\+
    \int_{\Gamma_h} \frac{1}{| \nabla \phi_h |} (z_h \cdot \nabla
    \phi_h + \frac{1}{h}q_h \phi_h) v_h
     + \gamma_{div} \| \Div z_h+v_h \|^2_{0, \Omega_h^{\Gamma}} + \gamma_u
    \|z_h + \nabla v_h \|^2_{0, \Omega_h^{\Gamma}} \\+ \sigma h \left\| \left[
    \ddn{v_h} \right] \right\|_{0, \Gamma_h^i}^2 + \frac{\gamma_p}{h^2} \| z_h
    \cdot \nabla \phi_h + \frac{1}{h}q_h \phi_h \|_{0, \Omega_h^{\Gamma}}^2.
  \end{multline*}
 Since   $B_h \subset \Omega_h^{\Gamma}$, we remark that
  the  integral of $v_h \Div z_h$ can be combined with that of $v_h$ on $\Omega_h^{\Gamma}$ to give
  \[ \|v_h \|_{0, \Omega_h^{\Gamma}}^2 + \int_{B_h} v_h \Div z_h \ge \int_{B_h} v_h \left(
     \Div z_h + v_h \right) \geq - \|v_h \|_{0, \Omega^{\Gamma}_h} \|
     \Div z_h + v_h \|_{0, \Omega^{\Gamma}_h}. \]
  We also use an inverse inequality from Lemma \ref{Prelim3} and the fact that $1/| \nabla \phi_h |$ is uniformly bounded by Assumption \ref{asm1}, to estimate
  \[ \left| \int_{\Gamma_h} \frac{1}{| \nabla \phi_h |} (z_h \cdot
     \nabla \phi_h +\frac{1}{h}q_h \phi_h) v_h \right| \leq \frac{C}{h}  \|z_h
     \cdot \nabla \phi_h + \frac{1}{h}q_h \phi_h \|_{0, \Omega_h^{\Gamma}} \|v_h \|_{0,
     \Omega_h^{\Gamma}}. \]
  Applying the Young inequality (for any $\varepsilon > 0$) to the last two bounds and combining this with \eqref{Neu:prop1}  yields
  \begin{multline*}
    a_h (v_h, z_h ,q_h; v_h, z_h,q_h) \geq (1 - \alpha) |v_h |_{1, \Omega_h}^2
    + \|v_h \|_{0, \Omega^i_h}^2 - (\varepsilon+\beta h^2) \|v_h \|^2_{0,
    \Omega_h^{\Gamma}}\\
    + \left( \gamma_{div} - \frac{1}{2 \varepsilon} - \beta h^2 \right) \|
    \Div z_h + v_h \|^2_{0, \Omega_h^{\Gamma}} + (\gamma_u - \delta)  \|z_h +
    \nabla v_h \|^2_{0, \Omega_h^{\Gamma}} \\+ (\sigma - \beta) h \left\| \left[
    \ddn{v_h} \right] \right\|_{0, \Gamma_h^i}^2
    + \left( \frac{\gamma_p}{h^2} - \frac{C^2}{2 \varepsilon h^2} \right)
    \|z_h \cdot \nabla \phi_h + \frac{1}{h}q_h \phi_h \|_{0, \Omega_h^{\Gamma}}^2.
%    \label{aNhest1}
  \end{multline*}

  To bound further from below the first 3 terms we note, using Lemma \ref{Prelim1} and the trace inverse inequality,
  $$ \|v_h \|^2_{0, \Omega_h^{\Gamma}} \leq C (h\|v_h \|^2_{0,
     \Gamma_h^i} + h^2 | v_h |^2_{1, \Omega_h^{\Gamma}}) \leq C (\|v_h
     \|^2_{0, \Omega_h^i} + h^2 | v_h |^2_{1, \Omega_h}) $$
  so that, introducing any $\kappa\geq 0$ and observing $h\le h_0:=\text{diam}(\Omega)$,
  \begin{multline*}
 (1 - \alpha) |v_h |_{1, \Omega_h}^2 + \|v_h \|_{0, \Omega^i_h}^2 -
     (\varepsilon +\beta h^2) \|v_h \|^2_{0, \Omega_h^{\Gamma}} \\
     \geq (1 - \alpha) |v_h
     |_{1, \Omega_h}^2 + \|v_h \|_{0, \Omega^i_h}^2 + \kappa \|v_h \|^2_{0,
     \Omega_h^{\Gamma}}- (\varepsilon +\beta h_0^2 + \kappa) \|v_h \|^2_{0, \Omega_h^{\Gamma}} \\
   \geq (1 - \alpha - C (\varepsilon  +\beta h_0^2+ \kappa) h_0^2) |v_h |_{1, \Omega_h}^2 +
     (1 - C (\varepsilon  +\beta h_0^2+ \kappa)) \|v_h \|_{0, \Omega^i_h}^2 + \kappa \|v_h \|^2_{0,
     \Omega_h^{\Gamma}} .
  \end	{multline*}
  Taking $\varepsilon, \kappa, \beta$ sufficiently small and $\gamma_u, \gamma_p,
  \gamma_{div}$ sufficiently big, gives the announced lower bound for $a_h (v_h, z_h, q_h ; v_h, z_h, q_h)$.
\end{proof}

\subsection{Proof of the $H^1$ error estimate in Theorem \ref{th:error}}

  Under the Theorem's assumptions, the solution to (\ref{Neu:P}) is indeed in
  $H^{k + 2} (\Omega)$ and it can be extended to a function $\tilde{u} \in
  H^{k + 2} (\Omega_h)$ such that $\tilde{u} = u$ on $\Omega$ and
  \begin{equation}\label{ureg}
 \| \tilde{u} \|_{k + 2, \Omega_h} \le C (\|f\|_{k, \Omega} +\|g\|_{k + 1
     / 2, \Gamma}) \leq \|f\|_{k, \Omega} .    
  \end{equation}
  Introduce $y = - \nabla \tilde{u}$ and $p = - \frac{h}{\phi}  y \cdot
  \nabla \phi $ on $\Omega_h^{\Gamma}$.
  Then, $y\in H^{k + 1} (\Omega_h^{\Gamma})$ and $p \in H^k (\Omega_h^{\Gamma})$ by Lemma \ref{lemma:hardy}. Moreover, 
  \begin{equation}\label{yreg}
  \| y \|_{k + 1, \Omega_h^\Gamma} \le C \| \tilde{u} \|_{k + 2, \Omega_h} 
  \leq C \|f\|_{k, \Omega}
 % \end{equation}
  \mbox{~~and~~}
%  \begin{equation}\label{preg}
  \| p \|_{k, \Omega_h^\Gamma} 
  \le Ch \| y \|_{k + 1, \Omega_h^\Gamma}  
  \leq Ch \|f\|_{k, \Omega} .
  \end{equation}
  Clearly, $\tilde{u}$, $y$, $p$ satisfy
  \begin{multline*}
    a_h (\tilde{u}, y, p ; v_h, z_h, q_h) = \int_{\Omega_h} \tilde{f} v_h +
    \gamma_{div}  \int_{\Omega_h^{\Gamma}} \tilde{f} (\Div z_h+v_h) +
    \frac{\gamma_p}{h^2}  \int_{\Omega_h^{\Gamma}} (y \cdot \nabla \phi_h + \frac{1}{h}p
    \phi_h)  (z_h \cdot \nabla \phi_h + \frac{1}{h}q_h \phi_h),  \\
    \quad \forall (v_h, z_h, q_h) \in W_h^{(k)}
  \end{multline*}
  with $\tilde{f} \assign - \Delta \tilde{u}+ \tilde u$. It entails a Galerkin
  orthogonality relation
  \begin{multline}
    a_h  (\tilde{u} - u_h, y - y_h, p - p_h ; v_h, z_h, q_h) =
    \int_{\Omega_h} (\tilde{f} - f) v_h + 
     \gamma_{div}  \int_{\Omega_h^{\Gamma}}  (\tilde{f} - f)( \Div z_h+v_h) \\
     +   \frac{\gamma_p}{h^2}  \int_{\Omega_h^{\Gamma}}
    (y \cdot \nabla \phi_h + \frac{1}{h}p \phi_h )  (z_h \cdot
    \nabla \phi_h + \frac{1}{h}q_h \phi_h),  
    \quad \forall (v_h, z_h, q_h) \in W_h^{(k)}.\label{Rob:GO}
  \end{multline}
  Introducing the standard nodal interpolation $I_h$ or, if necessary, a Cl\'ement interpolation (recall that $p$ is only in $H^1(\Omega_h^\Gamma)$ if $k=1$), we then have by 
  Proposition \ref{LemRob:coer}, 
  \begin{multline*}
   c\interleave u_h - I_h  \tilde{u}, y_h - I_h y, p_h - I_h p
     \interleave_h
      \leq \displaystyle\sup_{(v_h,z_h,q_h)\in W_h^{(k)}}
\frac{ a_h  (u_h-I_h\tilde{u},y_h-I_hy, p_h-I_hp ; v_h, z_h, q_h)}{\interleave v_h,z_h,q_h\interleave_h}\\
  \leq \sup_{(v_h,z_h,q_h)\in W_h^{(k)}}
\frac{I - II - III}{\interleave v_h,z_h,q_h\interleave_h},
\end{multline*}
  where
\begin{align*}
 I = a_h  (e_u,e_y,e_p ; v_h, z_h, q_h),~~~
 II &=\int_{\Omega_h} (\tilde{f} - f) v_h + \gamma_{div}  \int_{\Omega_h^{\Gamma}}
    (\tilde{f} - f)( \Div z_h+v_h), \\
  III&= \frac{\gamma_p}{h^2}  \int_{\Omega_h^{\Gamma}}
    (y \cdot \nabla \phi_h + \frac{1}{h}p \phi_h )  (z_h \cdot
    \nabla \phi_h + \frac{1}{h}q_h \phi_h),
\end{align*}
%\textcolor{red}{reprendre le format de separation de terme de phi FEM}
with $e_u=\tilde{u}-I_h\tilde{u},\  e_y=y-I_h\tilde{y}$ and $e_p=p-I_h\tilde{p}.$

We now estimate each term separately.
Recalling \eqref{altah}, we have
    \begin{multline*}
   I \leq \| e_u \|_{1, \Omega_h}\| v_h \|_{1, \Omega_h} +
\|\mbox{div} e_y\|_{0,B_h}\|v_h\|_{0,B_h}+\|e_y\|_{0,B_h}|v_h|_{1,B_h}\\
+\|\frac{1}{|\nabla \phi_h|}(e_y\cdot\nabla \phi_h+\frac{1}{h}e_p\phi_h)\|_{0,\Gamma_h}\|v_h\|_{0,\Gamma_h}
     + \gamma_{div} \| \Div e_y+e_u \|_{0, \Omega_h^{\Gamma}}\| \Div z_h+u_h \|_{0, \Omega_h^{\Gamma}} \\ + \gamma_u
    \|e_y + \nabla e_u \|_{0, \Omega_h^{\Gamma}} \|z_h + \nabla v_h \|_{0, \Omega_h^{\Gamma}} + \sigma h \left\| \left[
    \ddn{e_u} \right] \right\|_{0, \Gamma_h^i}\left\| \left[
    \ddn{v_h} \right] \right\|_{0, \Gamma_h^i} \\+ \frac{\gamma_p}{h^2} \| e_y
    \cdot \nabla \phi_h + \frac{1}{h}e_p \phi_h \|_{0, \Omega_h^{\Gamma}}\| z_h
    \cdot \nabla \phi_h + \frac{1}{h}q_h \phi_h \|_{0, \Omega_h^{\Gamma}}.
 \end{multline*}
Applying Lemma \ref{Prelim3} to the $L^2$ norms on $\Gamma_h$, recalling that  $1/|\nabla\phi_h|$ is uniformly bounded on $\Omega_h^\Gamma$ (cf. Assumption \ref{asm1}), and recombining the terms, we get
    \begin{equation*}
   I \leq
    C\left(
    \| e_u \|_{1, \Omega_h}^2+    \| e_y \|_{1, \Omega_h^{\Gamma}}^2+
      h \left\| \left[    \ddn{e_u} \right] \right\|_{0, \Gamma_h^i}^2
      +\frac{1}{h^2} \| e_y
    \cdot \nabla \phi_h + \frac{1}{h}e_p \phi_h \|_{0, \Omega_h^{\Gamma}}^2
    \right)^{1/2}\\
 \interleave v_h,z_h,q_h\interleave_h.
  \end{equation*}
  The usual interpolation estimates give
    \begin{equation*}\begin{array}{c}
    \| e_u \|_{1, \Omega_h}^2+    \| e_y \|_{1, \Omega_h^{\Gamma}}^2
      + h \left\| \left[    \ddn{e_u} \right] \right\|_{0, \Gamma_h^i}^2
      \leq Ch^{2k}(\|\tilde u\|_{k+1,\Omega_h}^2+\|y\|_{k+1,\Omega_h^{\Gamma}}^2)\,.
  \end{array} \end{equation*}
Moreover, recalling that  $|\nabla\phi_h|$ and $\frac{1}{h}|\phi_h|$ are uniformly bounded on $\Omega_h^\Gamma$, we get
       \begin{equation*}
 \frac{1}{h^2} \| e_y
    \cdot \nabla \phi_h + \frac{1}{h}e_p \phi_h \|_{0, \Omega_h^{\Gamma}}^2
    \leq  \frac{C}{h^2} \left(\| e_y\|_{0, \Omega_h^{\Gamma}}^2
   +\| e_p  \|_{0, \Omega_h^{\Gamma}}^2 \right)
    \leq Ch^{2k}(|y|_{k+1,\Omega_h^{\Gamma}}^2+\frac{1}{h^2}|p|_{k,\Omega_h^{\Gamma}}^2).
  \end{equation*}
  Thus, by regularity estimates (\ref{ureg}),
   % \begin{equation*}\begin{array}{c}
  $  I \leq
%  Ch^{k}(\|f\|_{k,\Omega}+\|\tilde g\|_{k+1,\Omega_h^{\Gamma}})
 Ch^{k}\|f\|_{k,\Omega}
 \interleave v_h,z_h,q_h\interleave_h.$
%  \end{array} \end{equation*}
 
 We now estimate the second term
 \begin{eqnarray*}
|II|
    &\leq& C(\|\tilde{f} - f\|_{0,\Omega_h}\|v_h\|_{0,\Omega_h}
    +\|\tilde{f} - f\|_{0,\Omega_h^{\Gamma}}\|\Div z_h+v_h\|_{0,\Omega_h^{\Gamma}})\\
   & \leq& C\|\tilde{f} - f\|_{0,\Omega_h}\interleave v_h,z_h,q_h\interleave_h
    \leq Ch^k\|f\|_{k,\Omega\cup\Omega_h}\interleave v_h,z_h,q_h\interleave_h\,.
\end{eqnarray*}
Indeed, thanks to Lemma \ref{Prelim15} and $f=\tilde f$ on $\Omega$, 
\begin{equation}\label{ftildef}
  \|\tilde{f} - f\|_{0,\Omega_h}
= \|\tilde{f} - f\|_{0,\Omega_h\backslash\Omega}
\leq Ch^k\|\tilde{f} - f\|_{k,\Omega_h\backslash\Omega}
\leq Ch^k\|f\|_{k,\Omega\cup\Omega_h}.
\end{equation}
Finally,
   % \begin{align*}
  $$   |III|
    \leq
    \frac{C}{h} \| y \cdot \nabla \phi_h + \frac{1}{h}p \phi_h \|_{0, \Omega_h^{\Gamma}}\interleave v_h,z_h,q_h\interleave_h$$
    %\end{align*}
 %  \[ \frac{1}{c} \interleave u_h - I_h  \tilde{u}, y_h - I_h y, p_h - I_h p
%     \interleave_h \leq \interleave u_{} - I_h  \tilde{u}, y - I_h y, p -
%     I_h p \interleave_h \]
%  \[ + \| \tilde{f} - f \|_{0, \Omega_h} + \frac{1}{h} \| y \cdot \nabla
%     \phi_h + p \phi_h - \tilde{g} | \nabla \phi_h | \|_{0, \Omega_h^{\Gamma}}.
%  \]
%  The first term $\interleave u_{} - I_h  \tilde{u}, y - I_h y, p - I_h p
%  \interleave_h$ can be mostly treated as in [nocut]\textcolor{red}{to do}, but there is a new term
%  which we detail :
%  \[ \frac{1}{h} \| (y - y_h) \cdot \nabla \phi_h + (p - p_h) \phi_h \|_{0,
%     \Omega_h^{\Gamma}} \leq \frac{C}{h} \| y - y_h \|_{0,
%     \Omega_h^{\Gamma}} + C \| p - p_h \|_{0, \Omega_h^{\Gamma}}, \]
%  since $\phi_h$ is of order $h$ on $\Omega_h^{\Gamma}$. Thus,
%  \[ \frac{1}{h} \| (y - y_h) \cdot \nabla \phi_h + (p - p_h) \phi_h \|_{0,
%     \Omega_h^{\Gamma}} \leq Ch^k (| y |_{k + 1, \Omega_h^{\Gamma}} + | p
%     |_{k, \Omega_h^{\Gamma}}), \]
%  so that we conclude
%  \[ \interleave u_{} - I_h  \tilde{u}, y - I_h y, p - I_h p \interleave_h
%     \leq Ch^k (| \tilde{u} |_{k + 1, \Omega_h} + | y |_{k + 1,
%     \Omega_h^{\Gamma}} + | p |_{k, \Omega_h^{\Gamma}}) .\]
%  The term $\| \tilde{f} - f \|_{0, \Omega_h}$ is also treated as in [nocut]\textcolor{red}{to do}.
%  It remains the last term
and, recalling $y \cdot \nabla \phi + \frac{1}{h}p\phi =0$ on $\Omega_h^\Gamma$,
  \begin{multline*}
   \frac{1}{h} \| y \cdot \nabla \phi_h + \frac{1}{h}p \phi_h  \|_{0, \Omega_h^{\Gamma}} = \frac{1}{h} \| y \cdot \nabla
     (\phi_h - \phi) + \frac{1}{h}p (\phi_h - \phi) \|_{0, \Omega_h^{\Gamma}} 
    \\
   \leq \frac{1}{h} \| y \|_{0, \Omega_h^{\Gamma}} \| \nabla (\phi_h -
     \phi) \|_{\infty} +  \frac{1}{h^2}\| p \|_{0, \Omega_h^{\Gamma}} \| \phi_h
     - \phi \|_{\infty}  
    \\
   \leq Ch^{k} (\| y \|_{0, \Omega_h^{\Gamma}} + \| p \|_{0,  \Omega_h^{\Gamma}} ) 
    \leq Ch^{k}\| f \|_{k, \Omega}  
  \end{multline*}
  by regularity estimates (\ref{yreg}).
  Note that the optimal order is achieved here since $\phi$ is assumed of regularity $C^{k+2}$  and it is approximated by  finite elements of degree at least $k + 1$.

Combining the estimate for the terms $I$--$III$ leads to
  $$
  \interleave u_h - I_h  \tilde{u}, y_h - I_h y, p_h - I_h p
     \interleave_h
      \leq
      Ch^{k} \| f \|_{k, \Omega\cup\Omega_h} ,
  $$
so that, by the triangle inequality together with interpolation estimate, we get 
  \begin{equation}\label{ApriTri}
  \interleave u_h - \tilde{u}, y_h - y, p_h - p \interleave_h
      \leq
      Ch^{k} \| f \|_{k, \Omega\cup\Omega_h}  .
  \end{equation}
  This implies the announced $H^1$ error estimate for $u-u_h$.

\subsection{Proof of the $L^2$ error estimate in Theorem \ref{th:error}}

Since $\Omega\subset\Omega_h$, we can
introduce $w : \Omega \to \mathbb{R}$ such that
\begin{equation*} %\label{Neu:z}
- \Delta w +w= u - u_h  \text{ in } \Omega,\quad
\frac{\partial w}{\partial n} = 0 \text{ on } \Gamma.
\end{equation*}
By elliptic regularity, $\|w\|_{2, \Omega} \leq C \|u - u_h \|_{0, \Omega}$.
Let $\tilde{w}$ be an extension of $w$ from $\Omega$ to $\Omega_h$
preserving the $H^2$ norm estimate and set $w_h = I_h  \tilde{w}$. We
observe
\begin{multline*}
\|u - u_h \|_{0, \Omega}^2 = \int_{\Omega} \nabla (u -u_h) \cdot \nabla (w - w_h)+\int_{\Omega} (u - u_h) (w - w_h)
+ \int_{\Omega} \nabla (u - u_h) \cdot \nabla w_h\\+ \int_{\Omega}  (u - u_h) w_h
\le  Ch^{k+1} \|f\|_{k,\Omega_h}  
|\tilde w  |_{2, \Omega_h} 
+ \left| \int_{\Omega} \nabla (u - u_h) \cdot \nabla w_h
+ \int_{\Omega}  (u - u_h) w_h\right|
\end{multline*}
by the already proven $H^1$ error estimate and interpolation estimates for $I_h  \tilde{w}$ (recall also $\Omega\subset\Omega_h$).
Taking $v_h=w_h$, $z_h=0$ and $q_h=0$ in the Galerkin orthogonality relation \eqref{Rob:GO},
we obtain,
%$$a_h  (\tilde{u} - u_h, y - y_h,p-p_h ; w_h, 0,0) =
%\int_{\Omega_h} (\tilde{f} - f) w_h
%+ \gamma_{div}  \int_{\Omega_h^{\Gamma}}
%(\tilde{f} - f)w_h.$$ 
thanks to (\ref{altah}),
\begin{multline*}
\int_{\Omega_h} \nabla (\tilde{u} - u_h) \cdot
\nabla w_h + \int_{\Omega_h} (\tilde{u} - u_h)w_h + \int_{B_h} (w_h \Div (y-y_h) + (y-y_h) \cdot \nabla w_h) \\
+\int_{\Gamma_h} \frac{1}{| \nabla \phi_h |} ((y-y_h) \cdot \nabla
\phi_h + \frac{1}{h}(p-p_h) \phi_h) w_h
+\gamma_{div}\int_{\Omega_h^{\Gamma}}(\Div(y-y_h)+\tilde u -u_h)w_h \\
+ \gamma_u  \int_{\Omega_h^{\Gamma}} ((y-y_h) + \nabla (\tilde u -u_h)) \cdot \nabla w_h
+  \sigma h \int_{\Gamma_h^i} \left[ \ddn{(\tilde u -u_h)} \right] \left[ \ddn{w_h} \right]  = (1+\gamma_{div})\int_{\Omega_h} (\tilde{f} - f) w_h.
\end{multline*}
Using the last relation in the bound for $\|u - u_h \|_{0, \Omega}^2$, we can
further bound it as  
\begin{multline*} 
\|u - u_h \|_{0, \Omega}^2 \leqslant  \ Ch^{k + 1}  \|f\|_{k, \Omega_h}  | \tilde{w} |_{2, \Omega_h}+   \left| \int_{\Omega_h \setminus \Omega} \nabla (\tilde{u} - u_h) \cdot
\nabla w_h + \int_{\Omega_h \setminus \Omega} (\tilde{u} - u_h) w_h \right|\\
+ \left| \int_{B_h} (w_h \Div (y - y_h) + (y - y_h) \cdot \nabla w_h)
\right| 
+ \left| \int_{\Gamma_h} \frac{1}{| \nabla \phi_h |} ((y - y_h)
\cdot \nabla \phi_h + \frac{1}{h}(p - p_h) \phi_h) w_h \right|\\ + \left| \gamma_{div} 
\int_{\Omega_h^{\Gamma}} (\Div (y - y_h) + \tilde{u} - u_h) w_h \right| 
+ \left| \gamma_u  \int_{\Omega_h^{\Gamma}} ((y - y_h) + \nabla (\tilde{u} -
u_h)) \cdot \nabla w_h \right|\\ + \left| \sigma h \int_{\Gamma_h^i} \left[
\ddn{(\tilde{u} - u_h)} \right] \left[ \ddn{w_h} \right] \right| 
+ (1 + \gamma_{div}) \left| \int_{\Omega_h} (\tilde{f} - f) w_h \right|  \\
\leqslant \ Ch^{k + 1}  \|f\|_{k, \Omega_h}  | \tilde{w} |_{2, \Omega_h} 
+ C \interleave \tilde{u} -   u_h, y - y_h, p - p_h \interleave_h 
\times\left( \| w_h \|_{1, \Omega_h
	\setminus \Omega} \right.\\\left.+ \| w_h \|_{1, \Omega_h^{\Gamma}} + h \| w_h \|_{0,
	\Gamma_h} + \sqrt{h} \|[\nabla w_h]\|_{0, \Gamma_h^i} \right) 
+ C \| \tilde{f} - f \|_{0, \Omega_h \setminus \Omega}  \| w_h \|_{1, \Omega_h
	\setminus \Omega} .
\end{multline*}
It remains to bound different norms of $w_h$ featuring in the estimate above. 
By Lemma \ref{Prelim1} and interpolation estimates
\begin{equation*}
\|w_h \|_{0, \Omega_h \setminus \Omega} \leq \| \tilde{w} - I_h  \tilde{w}
\|_{0, \Omega_h \setminus \Omega} + \| \tilde{w} \|_{0, \Omega_h \setminus
	\Omega}
 \leq Ch^2 | \tilde{w} |_{2, \Omega_h \setminus \Omega} + C \left( \sqrt{h}
\| \tilde{w} \|_{0, \Gamma} + h| \tilde{w} |_{1, \Omega_h \setminus \Omega}
\right) \leq C \sqrt{h} \| \tilde{w} \|_{2, \Omega_h}.
\end{equation*}
Similarly,
\begin{multline*}
\| \nabla w_h \|_{0, \Omega_h \setminus \Omega} \leq \| \nabla (\tilde{w} -
I_h  \tilde{w})\|_{0, \Omega_h \setminus \Omega} + \| \nabla \tilde{w}
\|_{0, \Omega_h \setminus \Omega}\\
 \leq Ch | \tilde{w} |_{2, \Omega_h \setminus \Omega} + C \left(
\sqrt{h} \| \nabla \tilde{w} \|_{0, \Gamma} + h| \nabla \tilde{w} |_{1,
	\Omega_h \setminus \Omega} \right) \leq C \sqrt{h} \| \tilde{w} \|_{2,
	\Omega_h} .
\end{multline*}
Analogous estimates also hold for $\|w_h \|_{1, \Omega_h^{\Gamma}}$. Moreover, by interpolation estimates,
$$ \|[\nabla w_h]\|_{0, \Gamma_h^i} = \|[\nabla (\tilde{w} - I_h \tilde{w})]\|_{0,
	\Gamma_h^i} \leqslant C \sqrt{h} | \tilde{w} |_{2, \Omega_h} $$
and, by Lemma \ref{Prelim3},
$$h \|w_h \|_{0, \Gamma_h} \leq C \sqrt{h} \|w_h \|_{0, \Omega_h^{\Gamma}}
\leq C \sqrt{h} \| \tilde{w} \|_{2, \Omega_h}
.$$ Hence,
\begin{align*}
\|u - u_h \|_{0, \Omega}^2 & \leq Ch^{k + 1}  \|f\|_{k, \Omega_h} 
| \tilde{w} |_{2, \Omega_h} + C
\sqrt{h} (\interleave \tilde{u} - u_h, y - y_h, p - p_h
\interleave _h+ \| \tilde{f} - f \|_{0, \Omega_h \setminus \Omega} ) \|
\tilde{w} \|_{2, \Omega_h} .
\end{align*}
This implies, by (\ref{ftildef}) and (\ref{ApriTri}),
$
\|u - u_h \|_{0, \Omega}^2  \leq Ch^{k + \frac 12}  \|f\|_{k, \Omega_h} \| \tilde{w} \|_{2, \Omega_h} 
$, 
which entails the announced error estimate in $L^2 (\Omega)$ since $ \| \tilde{w} \|_{2, \Omega_h} \le C\|u - u_h \|_{0, \Omega}$.

\section{Conditioning}

We are now going to prove that the condition number of the finite element
matrix associated to the bilinear form $a_h$ is of order $1 / h^2$.

\begin{theorem}%\label{prop:cond gp}
	Under Assumptions \ref{asm0}--\ref{asm2} and recalling that the mesh $\Th$ is quasi-uniform, the condition number defined by
	${\kappa} (\mathbf{A}) := \| \mathbf{A} \|_2 \|
	\mathbf{A}^{- 1} \|_2$ of the matrix $\mathbf{A}$ associated to the
	bilinear form $a_h$ on $W_h^{(k)}$ satisfies
	$ \kappa (\mathbf{A}) \le Ch^{- 2}$.
	Here, $\| \cdot \|_2$ stands for the matrix norm associated to the vector
	2-norm $| \cdot |_2$.
\end{theorem}

\begin{proof}
	The proof is divided into 4 steps:
	
	\textbf{Step 1.}
	We shall prove for all $q_h \in Q_h^{(k)}$
	\begin{equation}\label{eq:q_h}
	\|q_h \phi_h \|_{0, \Omega_h^{\Gamma}} \ge Ch \|q_h \|_{0,
		\Omega_h^{\Gamma}} .
	\end{equation}
	We have
	\begin{equation}
	\label{maxqhfh} \min_{T,q_h \neq 0, \phi_h \neq 0} \frac{\| q_h \phi_h \|_{0,
			T}}{h_T\| q_h \|_{0, T} \| \nabla \phi_h \|_{\infty, T}} \geqslant C,
	\end{equation}
	where the minimum is taken over all simplexes $T$ with $h_T=\textrm{diam}\,(T)$ satisfying the regularity assumptions and all polynomials $q_h$ of degree $\leqslant k$
	and $\phi_h$ of degree $\leqslant l$, with $\phi_h$ vanishing at at least one
	point on $T$. Note that this excludes $\| \nabla \phi_h \|_{\infty, T} = 0$
	because $\phi_h$ would then vanish identically on $T$. The
	minimum in (\ref{maxqhfh}) is indeed attained since, by homogeneity, it can be
	taken over the compact set $\| q_h \|_{0, T} = \| \nabla \phi_h \|_{\infty, T}
	= 1$ and simplexes with $h_T=1$. Hence, (\ref{maxqhfh}) is valid with some $C>0$. Applying (\ref{maxqhfh}) on any mesh element $T \in
	\mathcal{T}_h^{\Gamma}$ to any $q_h \in Q_h^{(k)}$ and $\phi_h$ approximation
	to $\phi$ satisfying Assumption \ref{asm1} leads to
	$ \| q_h \phi_h \|_{0, T} \geqslant Ch_T \frac{m}{2} \| q_h \|_{0, T} $. 
	Taking the square on both sides and summing over all $T \in	\mathcal{T}_h^{\Gamma}$ yields (\ref{eq:q_h}).
	
	\textbf{Step 2.}
	We shall prove for all $(v_h, z_h, q_h) \in W_h^{(k)}$
	\begin{eqnarray}
	a_h (v_h, z_h, q_h ; v_h, z_h, q_h) \ge c\|v_h, z_h, q_h \|_0^2
	\label{lemma:coer ah}
	\end{eqnarray}
	with $\|v_h, z_h, q_h \|_0^2 = \|v_h \|_{0, \Omega_h}^2 + \|z_h
	\|_{0, \Omega_h^{\Gamma}}^2 +  \|q_h \|_{0, \Omega_h^{\Gamma}}^2$.
	Indeed, by Lemma \ref{LemRob:coer},
	\begin{equation*}
	a_h (v_h, z_h, q_h ; v_h, z_h, q_h) \geq c \triple{v_h, z_h, q_h}_h^2 
	\geq c (| |v_h | |_{1, \Omega_h}^2 +\|z_h + \nabla v_h \|^2_{0,
		\Omega_h^{\Gamma}} +\|z_h \cdot \nabla \phi_h + \frac{1}{h}q_h \phi_h \|_{0,\Omega_h^{\Gamma}}^2).
	\end{equation*}
	We have assumed here (without loss of generality) $h \leq 1$.
	By Young's inequality with any $\epsilon_1 \in (0, 1)$,
	\begin{equation}\label{eq:z_h}
	\|z_h + \nabla v_h \|^2_{0, \Omega_h^{\Gamma}} = \|z_h \|^2_{0, \Omega_h^{\Gamma}} + \| \nabla v_h \|^2_{0,\Omega_h^{\Gamma}} + 2 (z_h, \nabla v_h)_{0, \Omega_h^{\Gamma}}
	 \ge (1 - \epsilon_1) \|z_h
	\|^2_{0, \Omega_h^{\Gamma}} - \frac{1 - \epsilon_1}{\epsilon_1}  \| \nabla
	v_h \|^2_{0, \Omega_h^{\Gamma}}.
	\end{equation}
	Similarly, for any $\epsilon_2 \in (0, 1)$, using that $\nabla \phi_h$ is
	uniformly bounded,
	\begin{equation}\label{eq:z_h2}
	\|z_h \cdot \nabla \phi_h +\frac{1}{h} q_h \phi_h \|_{0, \Omega_h^{\Gamma}}^2 \ge
	\frac{1 - \epsilon_2}{h^2}\|\phi_hq_h \|^2_{0, \Omega_h^{\Gamma}} - C \frac{1 -
		\epsilon_2}{\epsilon_2}  \|z_h \|^2_{0, \Omega_h^{\Gamma}} .
	\end{equation}
	%Using 
	%  \begin{equation}\label{eq:q_h2}
	%  \|\nabla \phi_h \|_{\infty, \Omega_h^{\Gamma}} \ge \|\nabla \phi\|_{\infty, \Omega_h^{\Gamma}} -\|\nabla \phi_h-\nabla \phi \|_{\infty, \Omega_h^{\Gamma}}\geq \|\nabla \phi \|_{\infty, \Omega_h^{\Gamma}}-Ch \|\phi \|_{W^2_{\infty}( \Omega_h^{\Gamma})},
	%  \end{equation}
	%Using the Hardy inequality of Lemma \ref{lemma:hardy} and Assumption \ref{asm1}, one has, 
	Thus, combining \eqref{eq:z_h}, \eqref{eq:z_h2} and \eqref{eq:q_h},
	\begin{multline*} a_h (v_h, z_h, q_h ; v_h, z_h, q_h) \\\geq c \left( \left( 1 - \frac{1
		- \epsilon_1}{\epsilon_1} \right) | |v_h | |_{1, \Omega_h}^2 + \left( 1 -
	\epsilon_1 - C \frac{1 - \epsilon_2}{\epsilon_2} \right) \|z_h \|^2_{0,
		\Omega_h^{\Gamma}} + C(1 - \epsilon_2)  \|q_h \|^2_{0,
		\Omega_h^{\Gamma}} \right) .
	\end{multline*}
	Taking $\epsilon_1, \epsilon_2$ close to 1, we get (\ref{lemma:coer ah}).
	
	\
	
	\textbf{Step 3.} We shall prove for all $(v_h, z_h, q_h) \in W_h^{(k)}$
	\begin{eqnarray}
	a_h (v_h, z_h, q_h ; v_h, z_h, q_h) \leq \frac{C}{h^2} \|v_h, z_h,
	q_h \|^2_0 . \label{lemma:cont ah}
	\end{eqnarray}
	By definition of $a_h$ (\ref{altah}) and Cauchy-Schwarz inequality,
	\begin{multline*}
	a_h (v_h, z_h, q_h ; v_h, z_h, q_h) \\\leq C \left( \|v_h \|^2_{1,
		\Omega_h} +\|z_h \|^2_{1, \Omega_h^{\Gamma}}
	+ h\left\|\left[\ddn{v_h}\right] \right\|^2_{0, \Gamma_h^i} 
	  + \frac{1}{h^2} \|z_h \cdot
	\nabla \phi_h \|_{0, \Omega_h^{\Gamma}}^2 + 
	\frac{1}{h^4} \|q_h \phi_h\|_{0, \Omega_h^{\Gamma}}^2 \right)  \,.
	\end{multline*}
	This leads to (\ref{lemma:cont ah}) thanks to inverse inequalities and to the fact that both $\nabla\phi_h$ and $\frac{1}{h}\phi_h$ are uniformly bounded on $\Omega_h^\Gamma$.
	
	\textbf{Step 4.} We combine (\ref{lemma:coer ah}) and (\ref{lemma:cont ah}), and observe that the norm $\|v_h, z_h, q_h \|_0$ is equivalent to the 2-norm of the vector representing $(v_h, z_h, q_h)$. This leads to the desired result as at the end of the proof of Theorem 4.1 from \cite{phiFEM}.
\end{proof}

\section{Numerical simulations}\label{sec:num}

In this section, we illustrate $\phi$-FEM on three different test cases, cf. Fig. \ref{fig:simu domains}, exploring the errors with respect to exact ``manufactured"  solutions.
The numerical results for the 1st test case (in 2D) confirm the predicted theoretical estimates (in fact, better than theoretically predicted convergence rate is observed for the $L^2$ error). In the 2nd test case (also in 2D), we show that the optimal convergence is recovered even 
  when the level-set function $\phi$ is less regular than assumed by the theory. Our method is also compared with CutFEM \cite{burman2} in the last case. Finally, a 3D example is given in the 3rd test case.

In the first test case, we will treat some examples with non-homogeneous Neumann condition (Fig. 2-8) and Robin condition (Fig. 9) thanks to the modification of the scheme given in Remark \ref{nonhom}.
In the two last test cases, we will consider homogeneous Neumann conditions.

  The surrounding domains $\mathcal{O}$ are always chosen as boxes aligned with the Cartesian coordinates and the background meshes $\Th^\mathcal{O}$ are obtained from uniform Cartesian grids, dividing the cells into the simplexes  (semi-cross meshes in 2D) . We always use the numerical  quadrature of a high enough order so that all the integrals in (\ref{Neu:Ph}) are computed exactly.

We have implemented $\phi$-FEM both in \texttt{FreeFEM} \cite{freefem} and in \texttt{multiphenics} \cite{multiphenics}. Both implementations give the same results in our 2D test cases and we present here only those obtained with \texttt{FreeFEM}. Unfortunately, it is not possible to fully implement $\phi$-FEM in 3D using \texttt{FreeFEM} since it does not provide the  tools for the jumps on inter-element faces. That is why the numerical results for our 3rd test case were produced using  \texttt{multiphenics} only. The implementation scripts can be consulted on GitHub.\footnote{\url{https://github.com/michelduprez/PhiFEM-Neumann}}
\begin{remark}
 Our method (\ref{Neu:Ph}) features ``mixed'' terms, such as $\gamma_{1} \int_{\Omega_h^{\Gamma}} (y_h + \nabla u_h) \cdot (z_h + \nabla v_h)$ involving $u_h,v_h$ defined on mesh $\Th$ and $y_h,z_h$ defined on mesh $\Th^{\Gamma}$, a submesh of $\Th$. Such integrals cannot be implemented in the current version of \texttt{FEniCS} since it requires all the finite elements involved in a problem to be defined on the same mesh. This is why we have turned to the \texttt{multiphenics} library, a spin-off of \texttt{FEniCS}, that does not have such a restriction. On the other hand,  \texttt{FreeFEM} features interpolations between meshes in a user-friendly manner. However, we have discovered that a straightforward implementation of (\ref{Neu:Ph}) in \texttt{FreeFEM} involving an implicit interpolation from $\Th$ to $\Th^\Gamma$ can lead to some spurious oscillations in the error curves. Much better results (reported below) are obtained if we introduce explicitly the interpolation matrix from $V_h^{(k)}$ to its restriction on $\Th^\Gamma$, using the \texttt{FreeFEM} function \texttt{interpolate}. 
\end{remark}

\begin{figure}[H]
\begin{center}
\includegraphics[width=5cm]{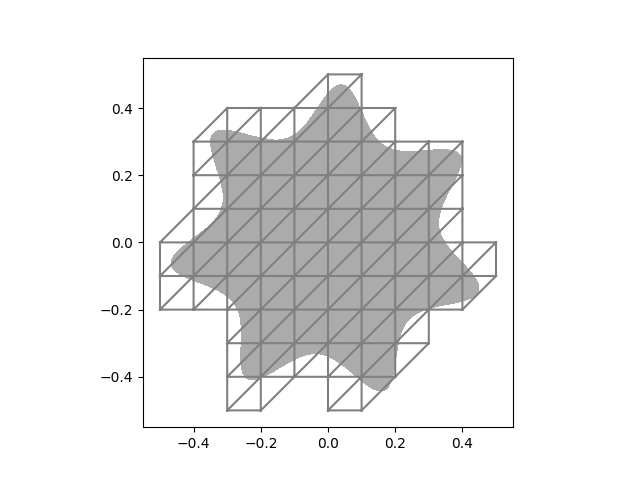} 
\includegraphics[width=5cm]{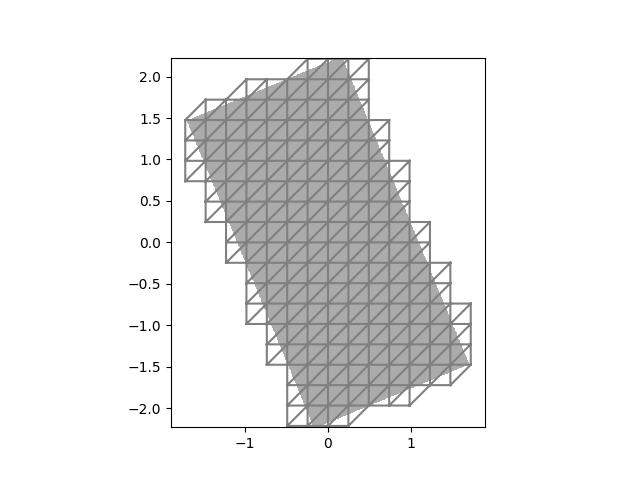} 
\includegraphics[width=5cm]{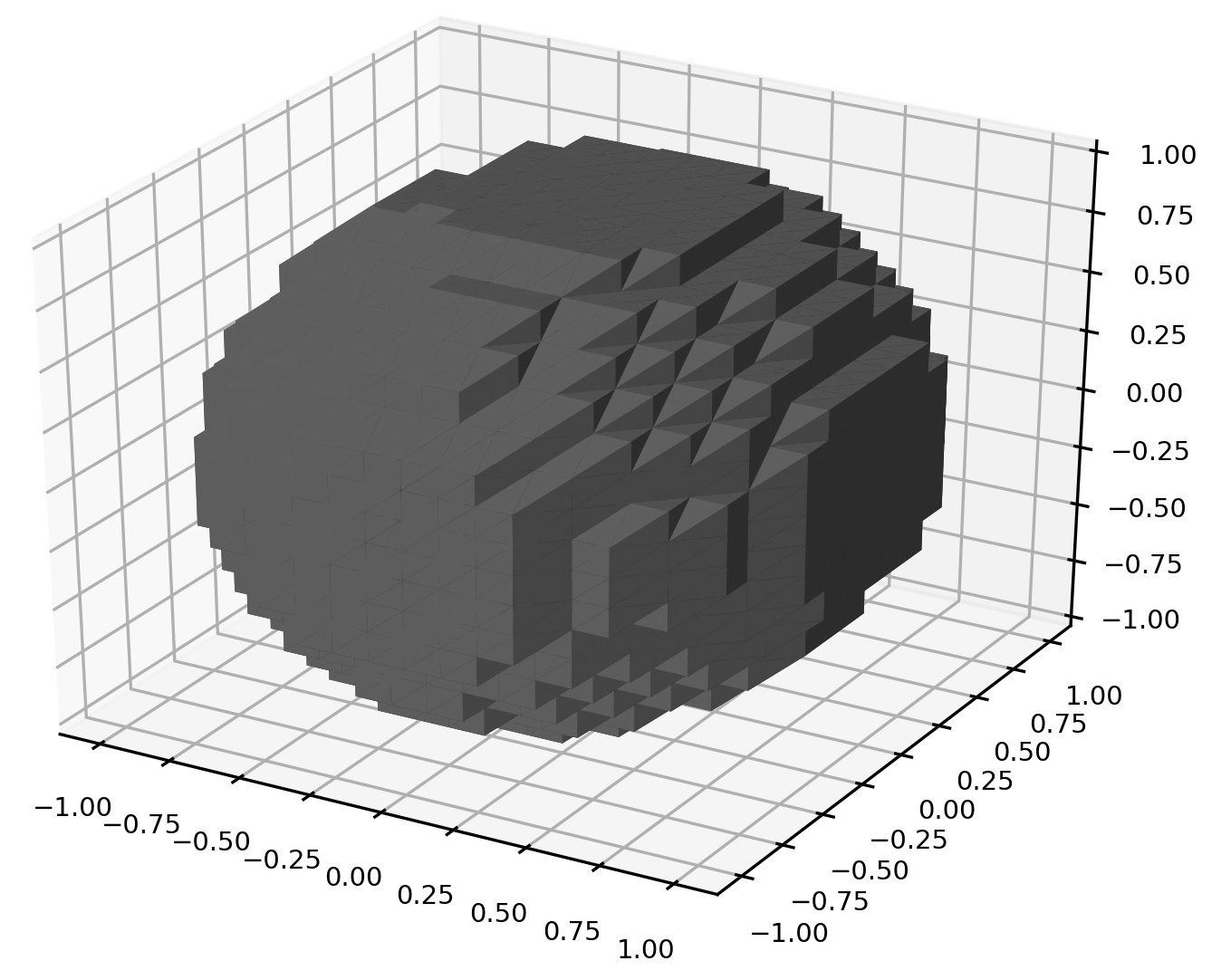} 
\caption{Domains and meshes considered in $\phi$-FEM for the test case 1 (left), test case 2 (center) and test case 3 (right). }
\label{fig:simu domains}
\end{center}\end{figure}

\subsection{1st test case}

Domain $\Omega$ (see Fig. \ref{fig:simu domains} left) is defined by the level-set function $\phi$ given in the polar coordinates $(r,\theta)$ by
\begin{equation*}%\label{eq:test1}
\phi(r,\theta)= r^4(5 + 3\sin(7(\theta-\theta_0) + 7\pi/36))/2 - R^4,
\end{equation*}
where $R=0.47$ and $\theta_0\in [0,2\pi)$. The surrounding domain $\mathcal{O}$ is fixed to $(-0.5,0.5)^2$. Varying the angle $\theta_0$ results in a rotation of $\Omega$, so that the boundary $\Gamma$ cuts the triangles of the background mesh in a different manner, creating sometimes the ''dangerous'' situations when certain mesh triangles of $\mathcal{T}_h$ have only a tiny portion inside the physical domain $\Omega$.

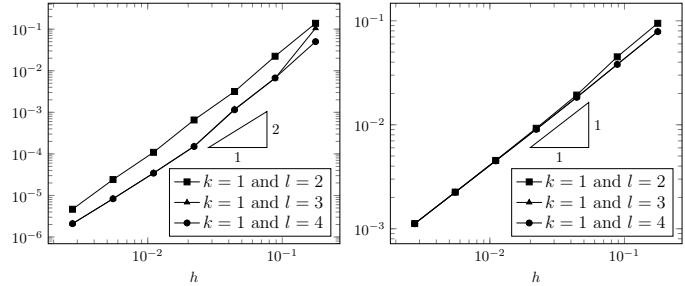
\begin{figure}[H]
\begin{center}
\begin{tikzpicture}[thick,scale=0.56, every node/.style={scale=1.0}]
\begin{loglogaxis}[xlabel=$h$%,xmin=3e-3,xmax=0.5 ,ymin=1e-6
%,legend style={at={(0.5,1.3)}
%legend pos=south west|south east|north west|north east
,legend pos=south east
,legend style={ font=\large}
, legend columns=1]%,ylabel=Gain]
 \addplot[%color=blue,
 mark=square*] coordinates {
(0.176777,0.13861)
(0.0883883,0.0223436)
(0.0441942,0.00316592)
(0.0220971,0.000654498)
(0.0110485,0.000109231)
(0.00552427,2.41225e-05)
(0.00276214,4.66225e-06)
};
 \addplot[%color=red,
 mark=triangle*] coordinates {
(0.176777,0.105681)
(0.0883883,0.00657242)
(0.0441942,0.0011822)
(0.0220971,0.000151746)
(0.0110485,3.45299e-05)
(0.00552427,8.32688e-06)
(0.00276214,2.09697e-06)
 };
 \addplot[%color=black,
 mark=*] coordinates {
(0.176777,0.0501426)
(0.0883883,0.00670024)
(0.0441942,0.00115167)
(0.0220971,0.000150767)
(0.0110485,3.44098e-05)
(0.00552427,8.32845e-06)
(0.00276214,2.09709e-06)
};

%\logLogSlopeTriangle{0.33}{0.2}{0.45}{1}{red};
\logLogSlopeTriangle{0.75}{0.2}{0.4}{2}{%blue
black};
 \legend{$k=1$ and $l=2$,$k=1$ and $l=3$,$k=1$ and $l=4$}
\end{loglogaxis}
\end{tikzpicture}
\begin{tikzpicture}[thick,scale=0.56, every node/.style={scale=1.0}]
\begin{loglogaxis}[xlabel=$h$%,xmin=3e-3,xmax=0.5 ,ymin=1e-6
%,legend style={at={(0.5,1.3)}
%legend pos=south west|south east|north west|north east
,legend pos=south east
,legend style={ font=\large}
, legend columns=1]%,ylabel=Gain]
 \addplot[%color=blue,
 mark=square*] coordinates {
(0.176777,0.0946733)
(0.0883883,0.0450185)
(0.0441942,0.0193182)
(0.0220971,0.00925605)
(0.0110485,0.00454114)
(0.00552427,0.0022504)
(0.00276214,0.00112029)
};
 \addplot[%color=red,
 mark=triangle*] coordinates {
(0.176777,0.0788302)
(0.0883883,0.0379922)
(0.0441942,0.0183564)
(0.0220971,0.00903373)
(0.0110485,0.0044924)
(0.00552427,0.00223979)
(0.00276214,0.00111832)
 };
 \addplot[%color=black,
 mark=*] coordinates {
(0.176777,0.078638)
(0.0883883,0.0380902)
(0.0441942,0.01836)
(0.0220971,0.00903369)
(0.0110485,0.00449241)
(0.00552427,0.00223979)
(0.00276214, 0.00111832)
 };

\logLogSlopeTriangle{0.68}{0.2}{0.4}{1}{%blue
black};
%\logLogSlopeTriangle{0.63}{0.2}{0.35}{2}{blue};
 \legend{$k=1$ and $l=2$,$k=1$ and $l=3$,$k=1$ and $l=4$}
\end{loglogaxis}
\end{tikzpicture}
\caption{$\phi$-FEM for the test case 1, %(see \eqref{eq:test1}--\eqref{eq:test1e})
 $\theta_0=0$,  $\sigma=0.01$ and $\gamma_u=\gamma_p=\gamma_{\mbox{div}}=10$, $k=1$ and different values of $l$. Left: $L^2$  relative error $\|u-u_h\|_{0,\Omega_h^i}/\|u\|_{0,\Omega_h^i}$; Right: $H^1$  relative error $\|u-u_h\|_{1,\Omega_h^i}/\|u\|_{1,\Omega_h^i}$. }
\label{fig:simu test 1 error L2 and H1}
\end{center}\end{figure}

We use $\phi$-FEM to solve numerically Poisson-Neumann problem \eqref{Neu:P} with non-homogeneous boundary conditions $\frac{\partial u}{\partial n}=g$ adjusting $f$ and $g$ so that the exact solution is given by
%\begin{equation}\label{eq:test1e}
%u(r,\theta)=\sin(r\cos(\theta-\theta_0))\exp(r\sin(\theta-\theta_0))
$u(x,y)=\sin(x)\exp(y).$
%\end{equation}
%with $\theta_0=0$.
The Neumann boundary condition is extrapolated to a vicinity of $\Gamma$ by 
%\begin{equation}\label{eq:testcase extrapol}
$\tilde g = \frac{\nabla u\cdot\nabla\phi}{|\nabla\phi|}+u \phi$, cf. Remark \ref{nonhom}.
%\end{equation}
The addition of $u \phi$ here does not perturb $\tilde g$ on $\Gamma$. Its purpose is to mimick the real life situation where $g$ is known on $\Gamma$ only and $\tilde g$ is some extension of $g$, not necessarily the natural one $\nabla u\cdot\nabla\phi/{|\nabla\phi|}$.

\begin{figure}[H]
\begin{center}
\begin{tikzpicture}[thick,scale=0.56, every node/.style={scale=1.0}]
\begin{loglogaxis}[xlabel=$h$%,xmin=3e-3,xmax=0.5 ,ymin=1e-6
%,legend style={at={(0.5,1.3)}
%legend pos=south west|south east|north west|north east
,legend pos=south east
,legend style={ font=\large}
, legend columns=1]%,ylabel=Gain]
 \addplot[%color=blue,
 mark=triangle*] coordinates {
(0.176777,0.00562683)
(0.0883883,0.00109056)
(0.0441942,3.53413e-05)
(0.0220971,1.3836e-05)
(0.0110485,8.83362e-07)
(0.00552427,9.48722e-08)
(0.00276214,3.52129e-09)

};
 \addplot[%color=red,
 mark=*] coordinates {
(0.176777,0.00746765)
(0.0883883,0.000119661)
(0.0441942,2.84064e-05)
(0.0220971,1.42383e-06)
(0.0110485,2.35584e-07)
(0.00552427,2.54233e-08)
(0.00276214,5.43098e-10)

 };

%\logLogSlopeTriangle{0.33}{0.2}{0.45}{1}{red};
\logLogSlopeTriangle{0.63}{0.2}{0.35}{3}{black%blue
};
 \legend{$k=2$ and $l=3$,$k=2$ and $l=4$}
\end{loglogaxis}
\end{tikzpicture}
\begin{tikzpicture}[thick,scale=0.56, every node/.style={scale=1.0}]
\begin{loglogaxis}[xlabel=$h$%,xmin=3e-3,xmax=0.5 ,ymin=1e-6
%,legend style={at={(0.5,1.3)}
%legend pos=south west|south east|north west|north east
,legend pos=south east
,legend style={ font=\large}
, legend columns=1]%,ylabel=Gain]
 \addplot[%color=blue,
 mark=triangle*] coordinates {
(0.176777,0.0059615)
(0.0883883,0.000800328)
(0.0441942,0.000139711)
(0.0220971,2.4863e-05)
(0.0110485,5.61744e-06)
(0.00552427,1.36194e-06)
(0.00276214,3.29611e-07)

};
 \addplot[%color=red,
 mark=*] coordinates {
(0.176777,0.0018531)
(0.0883883,0.000393935)
(0.0441942,8.6311e-05)
(0.0220971,2.07939e-05)
(0.0110485,5.13658e-06)
(0.00552427,1.27783e-06)
(0.00276214,3.18646e-07)

 };

\logLogSlopeTriangle{0.63}{0.2}{0.35}{2}{%blue
black};
%\logLogSlopeTriangle{0.63}{0.2}{0.35}{2}{blue};
 \legend{$k=2$ and $l=3$,$k=2$ and $l=4$}
\end{loglogaxis}
\end{tikzpicture}
\caption{$\phi$-FEM for the test case 1, %(see \eqref{eq:test1}--\eqref{eq:test1e})
 $\theta_0=0$,  $\sigma=0.01$ $\gamma_u=\gamma_p=\gamma_{\mbox{div}}=10$, $k=2$ and different values of $l$. Left: $L^2$  relative error $\|u-u_h\|_{0,\Omega_h^i}/\|u\|_{0,\Omega_h^i}$; Right: $H^1$  relative error $\|u-u_h\|_{1,\Omega_h^i}/\|u\|_{1,\Omega_h^i}$. }
\label{fig:simu test 1 error L2 and H1 P2}
\end{center}\end{figure}
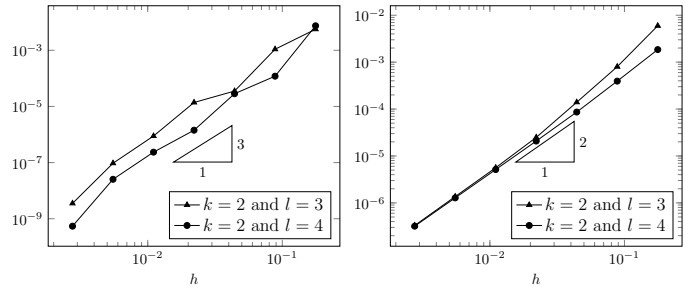

We report at Figs. \ref{fig:simu test 1 error L2 and H1} and \ref{fig:simu test 1 error L2 and H1 P2} the evolution of the relative error under the mesh refinement for a fixed position of $\Omega$ ($\theta_0=0$), using finite element spaces  $W_h^{(k)}$  with $k=1$ ($\mathbb{P}_1$ FE for $u_h$) and $k=2$ ($\mathbb{P}_2$ FE for $u_h$). We also try there different values of $l$, the degree of finite element used to approximate the level-set $\phi$, recalling that it should be chosen as $k+1$ or greater. The experiments reported in these figures confirms the optimal convergence order of the method in  both $H^1$  and $L^2$ norms (orders $k$ and $k+1$ respectively). The convergence order in the $L^2$ norm is thus better than in theory. An interesting experimental observation comes from exploring the degree $l$: while the lowest possible value $l=k+1$ ensures indeed the optimal convergence orders, it seems advantageous to increase the degree to  $l=k+2$, leading to more accurate results, especially in the $L^2$ norm.  Another series of experiments is reported at Figs. \ref{fig:simu test 1 phifem angle k=1 l=2} and \ref{fig:simu test 1 phifem angle k=1 l=3}. We explore there the errors with respect to the rotation of $\Omega$ over the background mesh (varying $\theta_0$). We restrict ourselves here with finite elements degree $k=1$ but compare two  different values of $l$: $l=k+1=2$ at Fig. \ref{fig:simu test 1 phifem angle k=1 l=2} vs. $l=k+2=3$ at Fig. \ref{fig:simu test 1 phifem angle k=1 l=3}. We observe again an advantage of the choice $l=k+2$: the oscillations on any given background mesh become less important when increasing $l$ and fade away under the mesh refinement in the case $l=k+2$ (this concerns mostly the $L^2$ errors; the $H^1$ errors are pretty much the same in both cases). The influence of the parameters $\sigma$, $\gamma_{div}$, $\gamma_u$, $\gamma_p$ on the accuracy of the method is explored by the numerical experiments reported at Figs. \ref{fig:simu test 1 phifem sensitivity sigma gamma} and \ref{fig:simu test 1 phifem sensitivity gamma}. Although a full assessment of the role of all the 4 parameters is difficult (we have chosen, somewhat arbitrarily, two scenarios of parameter variations out of endless other possibilities), the conclusion of our numerical experiments seems clear: the method is not sensible to variation of the parameters in the wide range from $10^{-6}$ to $10$, and there is no need to take these parameters greater than $10$.  Finally, we report at Fig. \ref{fig:simu test 1 phifem cond} evolution of the condition number of the $\phi$-FEM matrix under the mesh refinement and also its sensitivity with respect to the rotations of $\Omega$. The theoretically predicted behaviour of $\sim 1/h^2$ is confirmed. The conditioning of the method is also found to be rather insensitive to the position of $\Omega$ over the mesh.

\begin{figure}[H]
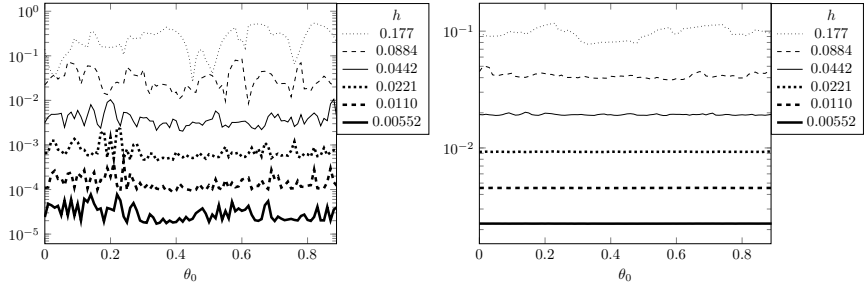

\begin{center}

\caption{Sensitivity of the relative error with respect to $\theta_0$  in $\phi$-FEM for the test case 1, %(see \eqref{eq:test1}--\eqref{eq:test1e})
 $\sigma=0.01$ and $\gamma_u=\gamma_p=\gamma_{\mbox{div}}=20$, $k=1$ and $l=2$. Left: $L^2$ relative error $\|u-u_h\|_{0,\Omega_h^i}/\|u\|_{0,\Omega_h^i}$; Right: $H^1$ relative error $\|u-u_h\|_{1,\Omega_h^i}/\|u\|_{1,\Omega_h^i}$.  }
\label{fig:simu test 1 phifem angle k=1 l=2}
\end{center}\end{figure}

\begin{figure}[H]
\begin{center}
%
\caption{Sensitivity of the relative error with respect to $\theta_0$  in $\phi$-FEM for the test case 1, %(see \eqref{eq:test1}--\eqref{eq:test1e})
 $\sigma=0.01$ and $\gamma_u=\gamma_p=\gamma_{\mbox{div}}=10$, $k=1$ and $l=3$. Left: $L^2$ relative error $\|u-u_h\|_{0,\Omega_h^i}/\|u\|_{0,\Omega_h^i}$; Right: $H^1$ relative error $\|u-u_h\|_{1,\Omega_h^i}/\|u\|_{1,\Omega_h^i}$.  }
\label{fig:simu test 1 phifem angle k=1 l=3}
\end{center}\end{figure}

\begin{figure}[H]
\begin{center}
%
\caption{Sensitivity of the relative error in $\phi$-FEM with respect to $\sigma$ with $\gamma_u=\gamma_p=\gamma_{div}=10$ being fixed for the test case 1, %(see \eqref{eq:test1}--\eqref{eq:test1e})
 $\theta_0=0$, $k=1$ and $l=3$. Left: $L^2$ relative error; Right: $H^1$ relative error.  }
\label{fig:simu test 1 phifem sensitivity sigma gamma}
\end{center}\end{figure}

\begin{figure}[H]
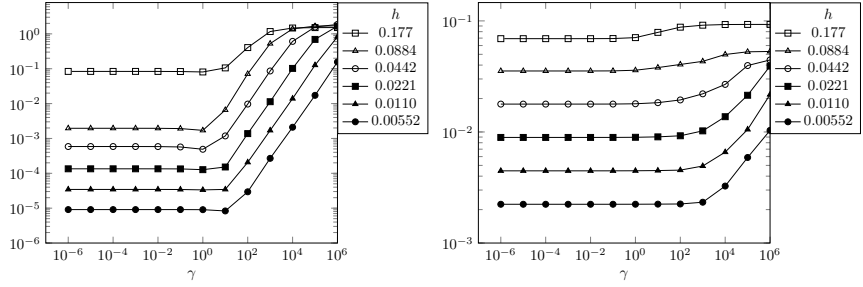

\begin{center}
%
\caption{Sensitivity of the relative error in $\phi$-FEM with respect to $\gamma_u=\gamma_p=\gamma_{div}=\gamma$ with $\sigma=0.01$ fixed  for the test case 1, %(see \eqref{eq:test1}--\eqref{eq:test1e})
 $\theta_0=0$, $k=1$ and $l=3$. Left: $L^2$ relative error; Right: $H^1$ relative error.  }
\label{fig:simu test 1 phifem sensitivity gamma}
\end{center}\end{figure}

%\FloatBarrier

\begin{figure}[H]
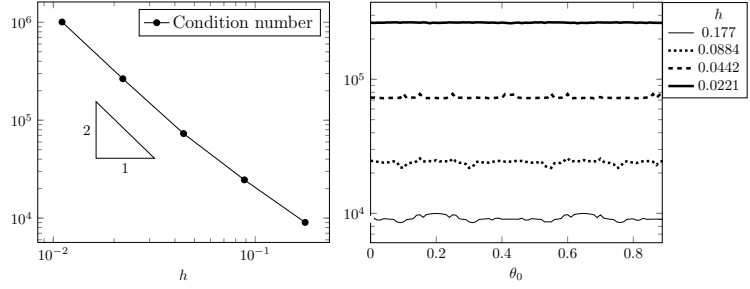

\begin{center}
%
\caption{Condition number in $\phi$-FEM for the test case 1, %(see \eqref{eq:test1}--\eqref{eq:test1e})
 $\sigma=0.01$, $\gamma_u=\gamma_p=\gamma_{\mbox{div}}=10$, $\theta_0=0$,  $k=1$ and $l=k+2$. Left: $\theta_0=0$; Right: different values of $\theta_0$. }
\label{fig:simu test 1 phifem cond}
\end{center}\end{figure}

%\FloatBarrier

We end this section by given an example with Robin boundary condition with $\alpha=1$ thanks to modification of main scheme presented in Remark \ref{nonhom}. We consider the same domain $\Omega$, level-set function $\phi$ and solution $u$ as before. The Robin condition is extrapolated  
by $\tilde g = \frac{\nabla u\cdot\nabla\phi}{|\nabla\phi|}+\alpha u + u \phi$.
In Fig. \ref{fig:simu test robin error L2 and H1}, we report the $L^2$ errors and the $H^1$ error (left) and the condition number (right) for $k=1$ and $l=3$. 
 We observe that optimal convergence order and standard condition number remain valid
for our Robin formulation.

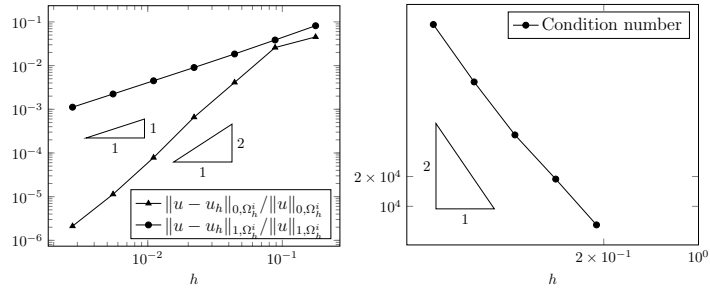
\begin{figure}[H]
\begin{center}
\begin{tikzpicture}[thick,scale=0.56, every node/.style={scale=1.0}]
\begin{loglogaxis}[xlabel=$h$%,xmin=3e-3,xmax=0.5 ,ymin=1e-6
%,legend style={at={(0.5,1.3)}
%legend pos=south west|south east|north west|north east
,legend pos=south east
,legend style={ font=\large}
, legend columns=1]%,ylabel=Gain]
 \addplot[%color=blue,
 mark=triangle*] coordinates {
(0.176776695297,0.0453178038217)
(0.0883883476483,0.0258891448454)
(0.0441941738242,0.004118084539)
(0.0220970869121,0.000658531106076)
(0.011048543456,7.83441097303e-05)
(0.00552427172802,1.13554704779e-05)
(0.00276213586401,2.11689786411e-06)
};
 \addplot[%color=red,
 mark=*] coordinates {
(0.176776695297,0.0817286149115)
(0.0883883476483,0.0385722426266)
(0.0441941738242,0.0184503004745)
(0.0220970869121,0.00903973109283)
(0.011048543456,0.00449221857781)
(0.00552427172802,0.00223973273223)
(0.00276213586401,0.00111829854053)
 };

\logLogSlopeTriangle{0.33}{0.2}{0.45}{1}{%red
black};
\logLogSlopeTriangle{0.63}{0.2}{0.35}{2}{%blue
black};
 \legend{$\|u-u_h\|_{0,\Omega_h^i}/\|u\|_{0,\Omega_h^i}$,$\|u-u_h\|_{1,\Omega_h^i}/\|u\|_{1,\Omega_h^i}$}
\end{loglogaxis}
\end{tikzpicture}
\begin{tikzpicture}[thick,scale=0.56, every node/.style={scale=1.0}]
\begin{loglogaxis}[xlabel=$h$,xmax=1e0,% xmin=1e-1,%ymin=5e3,ymax=5e4,
ytick={1e4,2e4},xtick={2e-1,1e0},
yticklabels={$10^{4}$,$2\times 10^{4}$},
xticklabels={$2\times 10^{-1}$,$10^{0}$}
%,legend style={at={(0.5,1.3)}
%legend pos=south west|south east|north west|north east
,legend pos=north east
,legend style={ font=\large}
, legend columns=1]%,ylabel=Gain]
 \addplot[%color=blue,
 mark=*] coordinates {
(0.176776695297,6498.80303698)
(0.0883883476483,18844.875535)
(0.0441941738242,52558.0180143)
(0.0220970869121,180063.682267)
(0.011048543456,683797.879967)

};

 \logLogSlopeTriangleinv{0.3}{0.2}{0.15}{2}{%blue
 black};
 \legend{Condition number}
\end{loglogaxis}
\end{tikzpicture}
\caption{$\phi$-FEM for the test case 1 and Robin boundary conditions, %(see \eqref{eq:test3 u}-\eqref{eq:test3 phi})
 $\sigma=0.01$, $\gamma_u=\gamma_p=\gamma_{\mbox{div}}=10$, $\alpha=1$, $k=1$ and $l=3$. Left: $L^2$ relative error $\|u-u_h\|_{0,\Omega_h^i}/\|u\|_{0,\Omega_h^i}$ and  $H^1$ relative error $\|u-u_h\|_{1,\Omega_h^i}/\|u\|_{1,\Omega_h^i}$. ; Right: Condition number.}
\label{fig:simu test robin error L2 and H1}
\end{center}\end{figure}

\subsection{2nd test case}
In this test case, the domain $\Omega$ is the rectangle $(-1,1)\times(-2,2)$ rotated by an angle $\theta_0$ counter-clockwise around the origin. It is defined by the level-set function $\phi$ given by
%\begin{equation}\label{eq:test2 phi}
$\phi(x,y)= \Phi\circ \Pi(x,y),$
%\end{equation}
with
%\begin{equation*}%\label{eq:test4}
$\Phi(x,y)= \max(|x|,|y|/2)-1$
%\mbox{ and }
and $
\Pi\left(\begin{array}{c}
x\\y
\end{array}\right)
=\left(\begin{array}{cc}
\cos(\theta_0)&-\sin(\theta_0)\\
\sin(\theta_0)&\cos(\theta_0)
\end{array}\right)
\left(\begin{array}{c}
x\\y
\end{array}\right).$
%\end{equation*}
The surrounding domain is taken as $\mathcal{O}=(-R,R)^2$, with $R=1.1\sqrt{5}$, cf. Fig. \ref{fig:simu domains} middle.  
%with $\theta_0=\pi/8$.{eq:test2 phi}{eq:test2 u}

We use $\phi$-FEM to solve numerically Poisson-Neumann problem \eqref{Neu:P} with the exact solution given by
%\eqref{eq:test1e}.
%\begin{equation}\label{eq:test2 u}
$u(x,y)= U\circ \Pi(x,y),$
%\end{equation}
where
%\begin{equation*}%\label{eq:test4}
$U(x,y)= \cos(\pi x)\cos(\pi y/2).$
%\end{equation*}

The results are presented at Figs. \ref{fig:simu test 2 error L2 and H1} (left) and \ref{fig:simu test 2 phifem angle k=1 l=3}, first choosing a fixed inclination angle $\theta_0=\pi/8$, and then varying $\theta_0$ from 0 to $2\pi/7$. The numerical tests show again the optimal convergence of $\phi$-FEM with $\mathbb{P}_1$ finite elements in the $L^2$ and $H^1$ norms, notwithstanding the fact that the level-set function $\phi$ is less regular than assumed in our theoretical results. Note that we have used here the FE of degree $l=3$ to represent the level-set, which is higher than the minimal degree $k+1=2$ suggested by the theory. The situation is here similar to that of the tests case 1-2: the implementation using the lower degree $l=2$ elements (not reported here) is also optimally convergent but turns out to be less robust than $l=3$ wih respect to the placement of $\Omega$ over the mesh (higher oscillations, especially in the $L^2$ error, when varying $\theta_0$).

We have also compared our method with CutFEM \cite{burman2}: Find $u_h\in V_h^{(k)}$ s.t.
\begin{equation*}
\displaystyle\int_{\Omega}\nabla u_h\cdot\nabla v_h+\displaystyle\int_{\Omega} u_h v_h
 +\sigma h\sum_{E\in \mathcal{F}^{\Gamma}} \int_{E} \left[ \ddn{u_h} \right] \left[ \ddn{v_h} \right] =\int_{\Omega}fv_h+\int_{\Gamma}gv_h~~\forall ~v_h\in V_h^{(k)},
\end{equation*}
where 
%\begin{equation*}
$\mathcal{F}^{\Gamma}=\{E(\mbox{internal facet of }\mathcal{T}_h)\mbox{ such that }\exists T\in\mathcal{T}_h:T\cap\Gamma\neq \varnothing \mbox{ and }E \in\partial T \}.$
%\end{equation*}

The results are reported at Figs. \ref{fig:simu test 2 error L2 and H1} (right, the simulation at fixed inclination angle $\theta_0$) and \ref{fig:simu test 1 cutfem angle} (simulations with the rotating domain $\Omega$). Comparing two parts of Fig. \ref{fig:simu test 2 error L2 and H1}, we conclude that $\phi$-FEM and CutFEM are both optimally convergent and produce  very similar results. However, looking closer at Figs. \ref{fig:simu test 2 phifem angle k=1 l=3} and \ref{fig:simu test 1 cutfem angle}, we can point out an advantage of the $\phi$-FEM over the CutFEM: the former seems more robust with respect to the position of $\Omega$ over the background mesh, the oscillations of the $L^2$ errors with rotating the domain are more pronounced for the latter method (the $H^1$ errors are almost the same in both cases).

%\textcolor{red}{$N$ cellule par cote du carre : N=8,16,32...}

\begin{figure}[H]
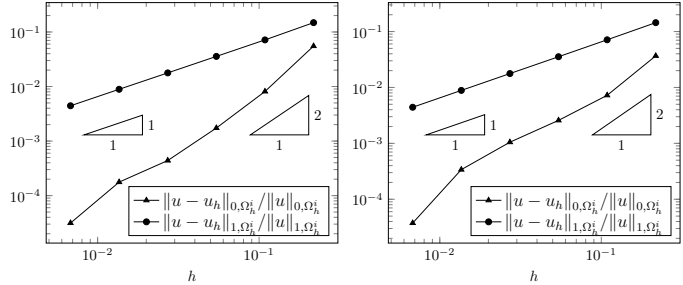

\begin{center}

\caption{$L^2$ and  $H^1$ relative error for the test case 2. %(see \eqref{eq:test2 phi}-\eqref{eq:test2 u}).
Left: $\phi$-FEM with $\sigma=0.01$, $\gamma_u=\gamma_p=\gamma_{\mbox{div}}=10$, $k=1$ and $l=3$; 
Right: CutFEM , $\theta_0=\pi/8$, $\sigma=0.01$. }
\label{fig:simu test 2 error L2 and H1}
\end{center}\end{figure}

\begin{figure}[H]
\begin{center}
%
\caption{Sensitivity of the relative error with respect to $\theta_0$  in $\phi$-FEM for the test case 2, %(see \eqref{eq:test2 phi}-\eqref{eq:test2 u}), % (see \eqref{eq:test1}--\eqref{eq:test1e}),
$\sigma=0.01$, $\gamma_u=\gamma_p=\gamma_{\mbox{div}}=10$, $k=1$ and $l=3$. Left: $L^2$ relative error $\|u-u_h\|_{0,\Omega_h^i}/\|u\|_{0,\Omega_h^i}$; Right: $H^1$ relative error $\|u-u_h\|_{1,\Omega_h^i}/\|u\|_{1,\Omega_h^i}$.  }
\label{fig:simu test 2 phifem angle k=1 l=3}
\end{center}\end{figure}

\begin{figure}[H]
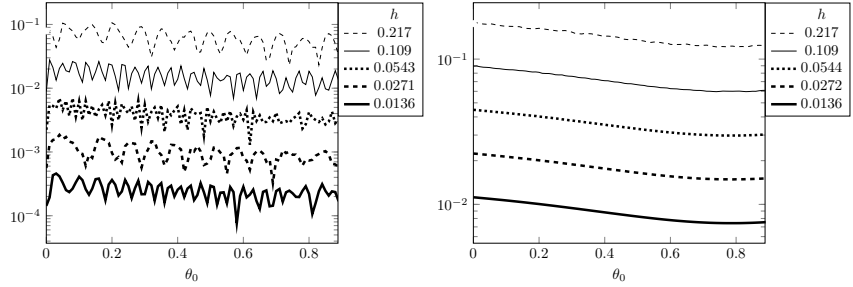

\begin{center}
%
\caption{Sensitivity of the relative error with respect to $\theta_0$  in CutFEM for the test case 2,  %(see \eqref{eq:test2 phi}-\eqref{eq:test2 u})
 $\sigma=0.01$. Left: $L^2$ relative error $\|u-u_h\|_{0,\Omega_h^i}/\|u\|_{0,\Omega_h^i}$; Right: $H^1$ relative error $\|u-u_h\|_{1,\Omega_h^i}/\|u\|_{1,\Omega_h^i}$.  }
\label{fig:simu test 1 cutfem angle}
\end{center}\end{figure}

%\FloatBarrier

 \subsection{3rd test case}
 
We here  take $\Omega\subset\mathbb{R}^3$ as the ball of radius $R=0.75$ centered at the origin encapsulated into the box $\mathcal{O}=(-1,1)^3$. $\Omega$ is defined by the level-set function 
% \begin{equation}\label{eq:test3 u}
 $\phi(x,y,z)=x^2+y^2+z^2-R^2.$
 % \end{equation}
 Fig. \ref{fig:simu domains} right gives an example of mesh $\Th$ for this test case.  
 %We choose $g=-sin(R)$, $R=0.4$ and 
 We choose the exact solution as
% \begin{equation}\label{eq:test3 phi}
$u(x,y,z)=\cos\left(\sqrt{x^2+y^2+z^2}\right).$
% \end{equation}
The Neumann boundary condition is extrapolated to a vicinity of $\Gamma$ as in the first case. %by \eqref{eq:testcase extrapol}.
Again, we observe in Fig. \ref{fig:simu test 3 error L2 and H1 and cond} the optimal orders of convergence for the $L^2$ and $H^1$ errors and the expexted behaviour of the condition number $\sim 1/h^2$.

%\textcolor{red}{$N$ cellule par cote du carre : N=4,8,16,32...}

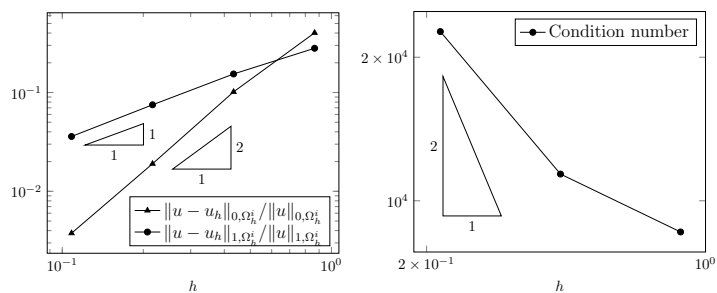
\begin{figure}[H]
\begin{center}
\begin{tikzpicture}[thick,scale=0.56, every node/.style={scale=1.0}]
\begin{loglogaxis}[xlabel=$h$%,xmin=3e-3,xmax=0.5 ,ymin=1e-6
%,legend style={at={(0.5,1.3)}
%legend pos=south west|south east|north west|north east
,legend pos=south east
,legend style={ font=\large}
, legend columns=1]%,ylabel=Gain]
 \addplot[%color=blue,
 mark=triangle*] coordinates {
(0.866025403784,0.4020093623844348)
(0.433012701892,0.1015394898345765)
(0.216506350946,0.01900553830763456)
(0.10825317547305482,0.0037637898177978035)
};
 \addplot[%color=red,
 mark=*] coordinates {
(0.866025403784,0.2803979220916107)
(0.433012701892,0.15404484985330788)
(0.216506350946,0.07509791128043267)
(0.10825317547305482,0.03588319632048886)
 };

\logLogSlopeTriangle{0.33}{0.2}{0.45}{1}{%red
black};
\logLogSlopeTriangle{0.63}{0.2}{0.35}{2}{%blue
black};
 \legend{$\|u-u_h\|_{0,\Omega_h^i}/\|u\|_{0,\Omega_h^i}$,$\|u-u_h\|_{1,\Omega_h^i}/\|u\|_{1,\Omega_h^i}$}
\end{loglogaxis}
\end{tikzpicture}
\begin{tikzpicture}[thick,scale=0.56, every node/.style={scale=1.0}]
\begin{loglogaxis}[xlabel=$h$,xmax=1e0,% xmin=1e-1,%ymin=5e3,ymax=5e4,
ytick={1e4,2e4},xtick={2e-1,1e0},
yticklabels={$10^{4}$,$2\times 10^{4}$},
xticklabels={$2\times 10^{-1}$,$10^{0}$}
%,legend style={at={(0.5,1.3)}
%legend pos=south west|south east|north west|north east
,legend pos=north east
,legend style={ font=\large}
, legend columns=1]%,ylabel=Gain]
 \addplot[%color=blue,
 mark=*] coordinates {
(0.866025403784,8602.79388462)
(0.433012701892,11377.5184058)
(0.216506350946,22638.0663319)

};

 \logLogSlopeTriangleinv{0.3}{0.2}{0.15}{2}{%blue
 black};
 \legend{Condition number}
\end{loglogaxis}
\end{tikzpicture}
\caption{$\phi$-FEM for the test case 3, %(see \eqref{eq:test3 u}-\eqref{eq:test3 phi})
 $\sigma=0.01$, $\gamma_u=\gamma_p=\gamma_{\mbox{div}}=10$, $k=1$ and $l=3$. Left: $L^2$   and  $H^1$  relative error; Right: condition number. }
\label{fig:simu test 3 error L2 and H1 and cond}
\end{center}\end{figure}

\bibliographystyle{abbrv}

%  This includes the bib file
\bibliography{biblio}

\end{document}